\documentclass[11pt]{amsart}
\usepackage[foot]{amsaddr}
\usepackage[english]{babel}

\allowdisplaybreaks

\usepackage[%pagebackref,
pdftex,hyperfootnotes]{hyperref}
\hypersetup{
 colorlinks=true,
 linkcolor=NavyBlue, 
 urlcolor=RoyalPurple,
 citecolor=OliveGreen,
 pdftitle={Oscillory banded matrices and mixed multiple orthogonality},
 bookmarks=true,
 pdfpagemode=FullScreen,
}
\usepackage[usenames,dvipsnames,svgnames,table,x11names]{xcolor}
\usepackage{pgfplots}
\usepackage{tikz}
\usepackage{tikz-3dplot}
\usetikzlibrary{automata,quotes, chains,matrix,calc,shadows,shapes.callouts,shapes.geometric,shapes.misc,positioning,patterns,decorations.shapes,
 decorations.pathmorphing,decorations.markings,decorations.fractals,decorations.pathreplacing,shadings,fadings,arrows.meta,bending}

\usepackage{verbatim}

\tikzstyle{block} = [draw, rectangle, 
minimum height=3em, minimum width=2em]
\usepackage{nicematrix}
\usepackage[utf8]{inputenc}

\usepackage{comment}
\usepackage[textwidth=17.5cm,textheight=22.275cm,
height=%24cm,
24.275cm,
width=18cm]{geometry}

\usepackage{amssymb,latexsym,amsmath,amsthm,bm}
\usepackage{mathrsfs}
\usepackage{mathtools,arydshln,mathdots}

\usepackage{tcolorbox}

\usepackage[table]{xcolor}

\usepackage{Baskervaldx}
\usepackage[]{newtxmath}

\usepackage{bigints}

\theoremstyle{plain}

\newtheorem{teo}{Theorem}[section]

\newtheorem{lemma}[teo]{Lemma}
\newtheorem{pro}[teo]{Proposition}

\theoremstyle{definition}
\newtheorem{defi}[teo]{Definition}

\theoremstyle{remark}
\newtheorem{rem}[teo]{Remark}

\renewcommand{\d}{\operatorname{d}}
\newcommand{\Exp}[1]{\operatorname{e}^{#1}}

\renewcommand{\Re}{\operatorname{Re}}

\newcommand{\N}{\mathbb{N}}
\newcommand{\R}{\mathbb{R}}
\newcommand{\Z}{\mathbb{Z}}

\newcommand*\pFqskip{8mu}
\catcode`,\active
\newcommand*\pFq{\begingroup
 \catcode`\,\active
 \def ,{\mskip\pFqskip\relax}%
 \dopFq
}
\catcode`\,12
\def\dopFq#1#2#3#4#5{%
 {}_{#1}F_{#2}\biggl[\genfrac..{0pt}{}{#3}{#4};#5\biggr]%
 \endgroup
}
\newcommand{\KF}[5]{F^{#1}_{#2}\left[{#3\atop #4}\Bigg\vert #5\right]}
\allowdisplaybreaks[1]

\makeatletter
\DeclareRobustCommand{\gaussk}{\DOTSB\gaussk@\slimits@}
\newcommand{\gaussk@}{\mathop{\vphantom{\sum}\mathpalette\bigcal@{K}}}

\newcommand{\bigcal@}[2]{%
 \vcenter{\m@th
  \sbox\z@{$#1\sum$}%
  \dimen@=\dimexpr\ht\z@+\dp\z@
  \hbox{\resizebox{!}{0.8\dimen@}{$\mathcal{K}$}}%
 }%
}
\newcommand{\cfracplus}{\mathbin{\cfracplus@}}
\newcommand{\cfracplus@}{%
 \sbox\z@{$\dfrac{1}{1}$}%
 \sbox\tw@{$+$}%
 \raisebox{\dimexpr\dp\tw@-\dp\z@\relax}{$+$}%
}
\newcommand{\cfracdots}{\mathord{\cfracdots@}}
\newcommand{\cfracdots@}{%
 \sbox\z@{$\dfrac{1}{1}$}%
 \sbox\tw@{$+$}%
 \raisebox{\dimexpr\dp\tw@-\dp\z@\relax}{$\cdots$}%
}
\makeatother

\makeatletter
\newcommand*{\relrelbarsep}{.386ex}
\newcommand*{\relrelbar}{%
 \mathrel{%
  \mathpalette\@relrelbar\relrelbarsep
 }%
}
\newcommand*{\@relrelbar}[2]{%
 \raise#2\hbox to 0pt{$\m@th#1\relbar$\hss}%
 \lower#2\hbox{$\m@th#1\relbar$}%
}
\providecommand*{\rightrightarrowsfill@}{%
 \arrowfill@\relrelbar\relrelbar\rightrightarrows
}
\providecommand*{\leftleftarrowsfill@}{%
 \arrowfill@\leftleftarrows\relrelbar\relrelbar
}
\providecommand*{\xrightrightarrows}[2][]{%
 \ext@arrow 0359\rightrightarrowsfill@{#1}{#2}%
}
\providecommand*{\xleftleftarrows}[2][]{%
 \ext@arrow 3095\leftleftarrowsfill@{#1}{#2}%
}
\makeatother
\begin{document}
 
 \title[Hahn multiple orthogonal polynomials of type I]{Hahn multiple orthogonal polynomials of type I: \\ Hypergeometrical expressions}

 \author[A Branquinho]{Amílcar Branquinho$^{1}$}
 \address{$^1$CMUC, Departamento de Matemática,
  Universidade de Coimbra, 3001-454 Coimbra, Portugal}
 \email{$^1$ajplb@mat.uc.pt}
 
 \author[JEF Díaz]{Juan E. F. Díaz$^{2}$}
 \email{$^2$juan.enri@ua.pt}
 \address{$^2$CIDMA, Departamento de Matemática, Universidade de Aveiro, 3810-193 Aveiro, Portugal}
 
 \author[A Foulquié]{Ana Foulquié-Moreno$^{3}$}
 \address{$^3$CIDMA, Departamento de Matemática, Universidade de Aveiro, 3810-193 Aveiro, Portugal}
 \email{$^3$foulquie@ua.pt}
 
 \author[M Mañas]{Manuel Mañas$^{4}$}
 \address{$^4$Departamento de Física Teórica, Universidad Complutense de Madrid, Plaza Ciencias 1, 28040-Madrid, Spain \&
  Instituto de Ciencias Matematicas (ICMAT), Campus de Cantoblanco UAM, 28049-Madrid, Spain}
 \email{$^4$manuel.manas@ucm.es}
 


 \begin{abstract}
 Explicit expressions for the Hahn multiple polynomials of type I, in terms of Kampé de Fériet hypergeometric series, are given.
 Orthogonal and biorthogonal relations are proven. Then, part of the Askey scheme for multiple orthogonal polynomials type I is completed. In particular, explicit expressions in terms of generalized hypergeometric series and Kampé de Fériet hypergeometric series, are given for the multiple orthogonal polynomials of type I for the Jacobi--Piñeiro, Meixner I, Meixner II, Kravchuk, Laguerre I, Laguerre II and Charlier families.
 
 \end{abstract}
 
 \maketitle
 
 % \clearpage
 
 \tableofcontents
\section{Introduction}

The study of multiple orthogonal polynomials and its applications is a very active area of research. However, although many multiple orthogonal polynomials of type II have already been explicitly found, see \cite{Arvesu}, \cite{AskeyII}, \cite{Clasicos}, \cite{Aptekarev} and the \cite[Chapter 23]{Ismail}, there is a lack for analogous expressions for multiple orthogonal polynomials of type I, and only very recently, in the last three years, some explicit expressions for the type I families have appeared, cf. \cite{JP}, \cite{masmultiples}, \cite{LL}, \cite{VanAssche_Wolfs}.

Discrete Hahn multiple orthogonal polynomials of type I is one of those families that lack such expressions. Hahn polynomials are relevant for several reasons, cf. \cite[Chapter 6]{Ismail} and \cite{Nikiforov_Suslov_Uvarov}, but for our interest in the study of multiple orthogonality is instrumental for being a kind of ascendant of many families of multiple orthogonal polynomials that are recovered from it by adequate limits, following the paths indicated by an Askey scheme. 
We have succeeded in the finding of such an explicit expression in terms of Kampé de Fériet series. These expressions hold for any couple of nonnegative indexes, not necessarily in the step-line. Consequently, through the application of the mentioned limits and
%following a partial 
inspired by the Askey scheme \cite{AskeyII}, we also succeed in finding hypergeometrical expressions for the multiple orthogonal polynomials of type I for the Jacobi--Piñeiro, Meixner I, Meixner II, Kravchuk, Laguerre I, Laguerre II and Charlier families.

 Multiple orthogonal polynomials naturally appear in the Hermite--Padé theory of simultaneous approximations and its different applications. Nowadays, we can find many texts with good introductions to the subject. We recommend to the interested reader the seminal book by Nikishin and Sorokin \cite{nikishin_sorokin} and the chapter 23 by
 Van Assche in
\cite%[Ch. 23]
{Ismail}.
For a basic and inspiring introduction we refer to \cite{andrei_walter}. Multiple orthogonal polynomials have had many achievements, here we just mention a few: for asymptotic of zeros see \cite{Aptekarev_Kaliaguine_Lopez}, for a Gauss--Borel perspective see \cite{afm}, for Christoffel perturbations see \cite{bfm}, for applications to random matrix theory see \cite{Bleher_Kuijlaars}, for Brownian bridges, or non-intersecting Brownian motions, that leave from $ p$ points and arrive to $q$ points see \cite{Evi_Arno}, for Markov chains see \cite{Hipergeometricos}, \cite{JP}, for multicomponent Toda, cf. \cite{adler,afm}, for applications to number theory see \cite{Apery,Ball_Rivoal,Zudilin}, 
for the spectral analysis of high order difference operators see \cite{Kalyagin,Kaliaguine,KaliaguineII,Aptekarev_Kaliaguine_Saff}. For applications to planar orthogonality  see \cite{Lee} and \cite{Bererzin_Kuijlaars_Parra}, in this last one type I multiple orthogonal polynomials on contours appear. See also \cite{PBF_1} and \cite{PBF_2} for spectral Favard theorems and the existence of systems positive Lebesgue--Stieltjes measures, associated with multiple orthogonal polynomials, for banded Hessenberg matrices and banded matrices, respectively, admitting a positive bidiagonal factorization. 
\enlargethispage{1cm}

\subsection{The Askey scheme}
The Askey scheme \cite{Askey} establishes a list of connections between hypergeometric orthogonal polynomials' families. From the scalar Hahn polynomials
\begin{align*}
\label{Hahnescalar}
 Q_n(x,\alpha,\beta,N)\equiv \dfrac{(\alpha+1)_n(-N)_n}{(\alpha+\beta+n+1)_n}\pFq{3}{2}{-n,-x, \alpha+\beta+n+1}{-N, \alpha+1}{1}
\end{align*}
there can be obtained all the following families through limit relations.
%\begin{figure}[H]

\begin{center}
  \begin{tikzpicture}[node distance=2cm]
  
 \node[fill=TealBlue!15,block] (a) {Jacobi}; 
 \node[block, fill=Peach!15,right of=a] (b) {Meixner}; 
 \node[fill=Peach!15,block,right of=b] (c) {Kravchuk};
 \node[block, fill=Peach!25,above of = b] (d) {Hahn};
 \node[fill=TealBlue!15,block, below of = a] (e) {Laguerre}; 
 \node[fill=TealBlue!15,block,right of=e] (f) {Hermite}; 
 \node[block,fill=Peach!15,right of=f] (g) {Charlier}; 
 \draw[-latex] (d)--(a); 
 \draw[-latex] (d)--(b); 
 \draw[-latex] (d)--(c);
 \draw[-latex] (a)--(e); 
 \draw[-latex] (a)--(f); 
 \draw[-latex] (b)--(e);
 \draw[-latex] (b)--(g);
 \draw[-latex] (c)--(f);
 \draw[-latex] (c)--(g);
 \draw[-latex] (e)--(f);
 \draw[-latex] (g)--(f);
\draw (2,-3) node[below] {\textbf{Askey Scheme}};
\end{tikzpicture} 
\end{center}
% \caption{Limit relations between hypergeometric scalar polynomials}
% \label{AskeyEscalar}
%\end{figure}
The Jacobi, Laguerre and Hermite polynomials are continuous, so they satisfy orthogonality relations of the form
\begin{equation*} \int_{\Delta}p_n(x)p_m(x)\omega(x)\d\mu(x)\propto\delta_{nm}
\end{equation*}
with respect to some weight function
$\omega:\Delta\subseteq\mathbb R\rightarrow\mathbb R$ and some measure $\mu:\Delta\subseteq\mathbb R\rightarrow\mathbb R^+$.
On the other hand, the Hahn, Meixner, Kravchuk and Charlier polynomials are discrete families. So they satisfy orthogonality relations of the form
\begin{align*}
 \sum_{k\in\Delta}p_n(k)p_m(k)\omega(k)\propto\delta_{mn}
\end{align*}
with respect to some weight function $\omega:\Delta\subseteq\mathbb Z\rightarrow\mathbb R$.

In \cite{Clasicos}, limit relations between Jacobi, Laguerre and Hermite families have been extended to the multiple case for type II polynomials. In \cite{Arvesu}, it has been done the same for the Hahn, Meixner, Kravchuk and Charlier families. In \cite{AskeyII} the Askey scheme has been almost completed for the type II multiple orthogonal polynomials. The main aim here is to get a scheme like the Askey one we have presented, for the type I multiple orthogonal polynomials.
We will construct explicit multiple hypergeometric series for the eight families of type I multiple orthogonal polynomials in the following list
\begin{enumerate}
 \item Hann
 \item Jacobi--Piñeiro
 \item Meixner I
 \item Meixner II
 \item Kravchuk
 \item Laguerre I
 \item Laguerre II
 \item Charlier
\end{enumerate}
and, as usual, we fail with multiple Hermite polynomials. Thus, we get for the type I situation the following Askey scheme diagram
\begin{center}
 \begin{tikzpicture}[node distance=3cm]
  \node[fill=TealBlue!15,block] (a) {Jacobi--Piñeiro}; 
  \node[ fill=Peach!15,block, right of=a] (b) {Meixner II};
  \node[ fill=Peach!15,block, right of = b] (c) {Meixner I};
  \node[ fill=Peach!15,block, right of=c] (f) {Kravchuk}; 
  \node[ fill=Peach!25,block,above of = c] (d) {Hahn};
  \node[fill=TealBlue!15,block, below of = a] (e) {Laguerre I};
  \node[fill=TealBlue!15,block, right of = e] (g) {Laguerre II};
  \node[ fill=Peach!15,block, below of =f] (i) {Charlier};
  \draw[-latex] (d)--(a); 
  \draw[-latex] (d)--(b);
  \draw[-latex] (d)--(c);
  \draw[-latex] (d)--(f);
  \draw[-latex] (a)--(e);
  \draw[-latex] (b)--(e);
  \draw[-latex] (c)--(i);
  \draw[-latex] (c)--(g);
  \draw[-latex] (f)--(i);
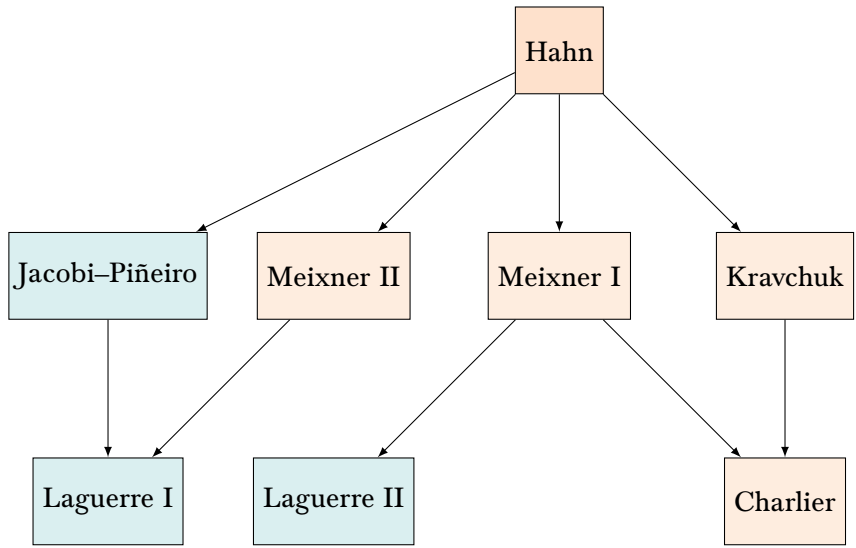
  \draw (5,-4) node[below] {\textbf{Multiple Askey Scheme}};
 \end{tikzpicture} 
\end{center}

\subsection{Generalizations of the hypergeometric series}
All of the scalar polynomials in the previous Askey scheme are given in terms of generalized hypergeometric functions \cite{LibrodeHypergeom, Srivastava_Karlsson,LibrodeKF}
\begin{align}
 \label{Hypergeom}
 \pFq{p}{q}{a_1,\dots,a_p}{b_1,\dots,b_q}{x}\coloneq\sum_{l=0}^{\infty}\dfrac{(a_1)_l\cdots(a_p)_l}{(b_1)_l\cdots(b_q)_l}\dfrac{x^l}{l!}.
\end{align}
In this paper we give explicit expressions of different families of multiple orthogonal polynomials of type I. Instead of these generalized hypergeometric series, the multiple orthogonal polynomials are frequently given in terms of Kampé de Fériet series \cite{LibrodeKF,Srivastava_Karlsson}
\begin{multline}
\label{KF}
 \KF{p:r;s}{q:n;k}{(a_1,\dots,a_p):(b_1,\dots,b_r);(c_1,\dots,c_s)}{(\alpha_1,\dots,\alpha_q):(\beta_1,\dots,\beta_n);(\gamma_1,\dots,\gamma_k)}{x,y}\\
 \coloneq\sum_{l=0}^{\infty}\sum_{m=0}^{\infty}\dfrac{(a_1)_{l+m}\cdots(a_p)_{l+m}}{(\alpha_1)_{l+m}\cdots(\alpha_q)_{l+m}}\dfrac{(b_1)_l\cdots(b_r)_l}{(\beta_1)_l\cdots(\beta_n)_l}\dfrac{(c_1)_m\cdots(c_s)_m}{(\gamma_1)_m\cdots(\gamma_k)_m}\dfrac{x^l}{l!}\dfrac{y^m}{m!}.
\end{multline}
Generalized hypergeometric and Kampé de Feriet hypergeometric series depend on the Pochhammer symbols
\begin{align*}
 (x)_n\coloneq\dfrac{\Gamma(x+n)}{\Gamma(x)}=\begin{cases}
  x(x+1)\cdots(x+n-1), &n\in\N,\\
  1, & n=0,
 \end{cases}
\end{align*}
which satisfy the following important property
\begin{align}
 \label{enteroneg}
 (-m)_n& =0, & m&<n, & m&\in\mathbb N_0,
\end{align}
This implies that if any of the $a_1,\dots,a_p$ in Equations \eqref{Hypergeom} and \eqref{KF} is a negative integer, then the series truncations and are finite summations. This is what happens in all the polynomials we are going to study later.

One important result concerning generalized hypergeometric series that will be useful is the following theorem, cf. \cite[Theorem 2.2]{KP}
\begin{teo}[Karp and Prilepkina]
\label{Karp}
Let be $r\in\mathbb N_0$, $a,b,f_1,\dots,f_r\in\mathbb C$, $p\in\mathbb N$, $m_1,\dots,m_r\in\mathbb N_0$. If $\Re{(p-a-m_1-\cdots-m_r)}>0$, then
\begin{multline*}
 \pFq{r+2}{r+1}{a,b,f_1+m_1,\dots,f_r+m_r}{b+p,f_1,\dots,f_r}{1}=\dfrac{\Gamma(1-a)\Gamma(b+p)}{(p-1)!\Gamma(b-a+1)}\dfrac{(f_1-b)_{m_1}\cdots(f_r-b)_{m_r}}{(f_1)_{m_1}\cdots(f_r)_{m_r}}\\
 \times\pFq{r+2}{r+1}{-p+1,b,-f_1+b+1,\dots,-f_r+b+1}{b-a+1,-f_1+b+1-m_1,\dots,-f_r+b+1-m_r}{1}.
\end{multline*}
\end{teo}

\subsection{Multiple orthogonal polynomials}

Here we follow Van Assche in \cite[Chapter 23]{Ismail}, and we assume from now on that the weights $(\omega_1,\omega_2)$ are an algebraic Chebyshev (AT) system and, consequently,  any index $(n_1,n_2)\in\N_0^2$  is normal and the system of weights $(\omega_1,\omega_2)$ is perfect. It is worthy to say that the systems of measures that we consider in this paper are known to be AT systems (cf. \cite{Aptekarev,Arvesu}).

 For the discrete multiple orthogonality, given a multi-index $(n_1,n_2)$ and two weight functions $\omega_1,\omega_2:\Delta\subseteq\Z\rightarrow\R$ we 
 %will deal with 
consider three polynomial $B_{(n_1,n_2)}$ with $\deg{B_{(n_1,n_2)}\leq n_1+n_2}$,
$\big\{A_{(n_1,n_2)}^{(1)}, A_{(n_1,n_2)}^{(2)}\big\}$ 
with $\deg{A_{(n_1,n_2)}^{(a)}\leq n_a-1}$,
and corresponding linear forms
\begin{align*}
\mathscr  Q_{(n_1,n_2)}\coloneq A_{(n_1,n_2)}^{(1)}(k)\omega_1(k)+A_{(n_1,n_2)}^{(2)}(k)\omega_2(k),
\end{align*}
 satisfying discrete orthogonality relations of the form
\begin{align}
 \label{ortdisII}
 \sum_{k\in\Delta}(-k)_jB_{(n_1,n_2)}(k)\omega_a(k)&=0, & j&\in\{0,\dots,n_a-1\},& a&\in\{1,2\}, \\
 \label{ortdiscI}
 \sum_{k\in\Delta}(-k)_j\mathscr Q_{(n_1,n_2)}(k)&=0, & j&\in\{0,\dots,n_1+n_2-2\}.
\end{align}
Analogously, for the continuous multiple orthogonality, given the couple $(n_1,n_2)$,
%and
two weight functions $\omega_1,\omega_2:\Delta\subseteq\mathbb R\rightarrow\mathbb R$, and a measure $\mu:\Delta\subseteq\mathbb R\rightarrow\mathbb R^+$, multiple orthogonality reads
\begin{align}
 \label{ortcontII}
 \int_{\Delta}x^jB_{(n_1,n_2)}(x)\omega_a(x)\d\mu(x)&=0, & j&\in\{0,\dots,n_a-1\},& a&\in\{1,2\}, \\
 \label{ortcontI}
 \int_{\Delta}x^j\mathscr Q_{(n_1,n_2)}(x)\d\mu(x)&=0, & j&\in\{0,\dots,n_1+n_2-2\}.
\end{align}
If this happens, the polynomials $B_{(n_1,n_2)}(x)$ and $\big\{ A^{(1)}_{(n_1,n_2)}(x), A^{(2)}_{(n_1,n_2)}(x)\big\}%_{a=1,2}
$ are said to be the multiple orthogonal polynomials of type II and type I respectively, respect to the given system of weight functions and measure.

Type I and II multiple orthogonal polynomials fulfill extended type biorthogonality relations of the form:
\begin{align}
 \label{biortogonalidaddisc}
 &\sum_{k\in\Delta}\mathscr Q_{(n_1,n_2)}(k)B_{(m_1,m_2)}(k)=
 \begin{cases}
  0,&\text{if $n_1\leq m_1, n_2\leq m_2$,}\\
  1,&\text{if $m_1+m_2=n_1+n_2-1$,}\\
  0,&\text{if $m_1+m_2\leq n_1+n_2-2$,}
 \end{cases}
\end{align}
in the discrete case and
\begin{align}
 \label{biortogonalidadcont}
 &\int_{\Delta}\mathscr Q_{(n_1,n_2)}(x)B_{(m_1,m_2)}(x)\d\mu(x)=
 \begin{cases}
 0,&\text{if $n_1\leq m_1, n_2\leq m_2$,}\\
 1,&\text{if $m_1+m_2=n_1+n_2-1$,}\\
 0,&\text{if $m_1+m_2\leq n_1+n_2-2$,}
\end{cases}
\end{align}
in the continuous case.

\begin{defi}[Near neighbor recursion coefficients]\label{def:near_neighbors_recursion_coeffcients}
Following \cite[Equations (23.1.13) and (23.1.14)]{Ismail}, let us introduce the following near neighbor recursion coefficients:
 \begin{gather*}
  \begin{aligned}
   b^{(1)}{(n_1,n_2)}&\coloneq\int x B_{(n_1,n_2)}(x)\mathscr Q_{(n_1+1,n_2)}(x)\d\mu(x),&
   b^{(2)}{(n_1,n_2)}&\coloneq\int x B_{(n_1,n_2)}(x)\mathscr Q_{(n_1,n_2+1)}(x)\d\mu(x),
  \end{aligned}\\
  c(n_1,n_2)\coloneq\int x B_{(n_1,n_2)}(x)\mathscr Q_{(n_1,n_2)}(x)\d\mu(x),\\
  \begin{aligned}
   d^{(1)}(n_1,n_2)&\coloneq\int x B_{(n_1,n_2)}(x)\mathscr Q_{(n_1-1,n_2)}(x)\d\mu(x),&
   d^{(2)}(n_1,n_2)&\coloneq\int x B_{(n_1,n_2)}(x)\mathscr Q_{(n_1,n_2-1)}(x)\d\mu(x).
  \end{aligned}
 \end{gather*}
\end{defi}
\begin{rem}
 Observe that the interchange of weights $\omega_1 \leftrightarrow\omega_2$ corresponds
 to $n_1 \leftrightarrow n_2$, $b^{(1)} \leftrightarrow b^{(2)}$ and $d^{(1)} \leftrightarrow d^{(2)}$ .
\end{rem}

We can restate \cite[Theorem 23.1.7]{Ismail} to our particular situation:

\begin{teo}[Type II near neighbor recursion relations]\label{Th:Recursion_Relations_II}
For perfect systems of weights the type II multiple orthogonal polynomials are subject to the following near neighbor recursion relations
\begin{align*}
xB_{(n_1,n_2)}&= B_{(n_1+1,n_2)}+b^{(1)}{(n_1,n_2)}B_{(n_1,n_2)}+c(n_1,n_2)B_{(n_1-1,n_2)}+d^{(1)}(n_1,n_2)B_{(n_1-1,n_2-1)},\\
xB_{(n_1,n_2)}&= B_{(n_1,n_2+1)}+ b^{(2)}{(n_1,n_2)}B_{(n_1,n_2)}+ c(n_1,n_2)B_{(n_1-1,n_2)}+ d^{(1)}(n_1,n_2)B_{(n_1-1,n_2-1)},\\
xB_{(n_1,n_2)}&= B_{(n_1+1,n_2)}+ b^{(1)}{(n_1,n_2)}B_{(n_1,n_2)}+ c(n_1,n_2)B_{(n_1,n_2-1)}+ d^{(2)}(n_1,n_2)B_{(n_1-1,n_2-1)},\\
xB_{(n_1,n_2)}&= B_{(n_1,n_2+1)}+b^{(2)}{(n_1,n_2)}B_{(n_1,n_2)}+c(n_1,n_2)B_{(n_1,n_2-1)}+d^{(2)}(n_1,n_2)B_{(n_1-1,n_2-1)},\\
\end{align*}
In particular,
\begin{align*}
 B_{(n_1+1,n_2)}- B_{(n_1,n_2+1)}=\big( b^{(2)}{(n_1,n_2)}-b^{(1)}{(n_1,n_2)}\big)B_{(n_1,n_2)}.
\end{align*}
\end{teo}

Similarly, for the linear forms of orthogonal polynomials of type I from the application of \cite[Theorem 23.1.9, Equations (23.1.17) and (23.1.18)]{Ismail} leads to:
\begin{teo}[Type I near neighbor recursion relations]\label{Th:Recursion_Relations_I}
 For perfect systems of weights, the linear forms and type I multiple orthogonal polynomials fulfill the following near neighbor recursion relations
\begin{align}\label{eq:Recursion_Relations_I.1}
 x\mathscr Q_{(n_1,n_2)}&= \mathscr Q_{(n_1-1,n_2)}+b^{(1)}{(n_1-1,n_2)}\mathscr Q_{(n_1,n_2)}+c(n_1,n_2)\mathscr Q_{(n_1+1,n_2)}+d^{(1)}(n_1+1,n_2)\mathscr Q_{(n_1+1,n_2+1)},\\\label{eq:Recursion_Relations_I.2}
  x\mathscr Q_{(n_1,n_2)}&= \mathscr Q_{(n_1,n_2-1)}+b^{(2)}{(n_1,n_2-1)}\mathscr Q_{(n_1,n_2)}+c(n_1,n_2)\mathscr Q_{(n_1+1,n_2)}+d^{(1)}(n_1+1,n_2)\mathscr Q_{(n_1+1,n_2+1)},\\\label{eq:Recursion_Relations_I.3}
  x\mathscr Q_{(n_1,n_2)}&= \mathscr Q_{(n_1-1,n_2)}+b^{(1)}{(n_1-1,n_2)}\mathscr Q_{(n_1,n_2)}+c(n_1,n_2)\mathscr Q_{(n_1,n_2+1)}+d^{(2)}(n_1,n_2+1)\mathscr Q_{(n_1+1,n_2+1)},\\ \label{eq:Recursion_Relations_I.4}
   x\mathscr Q_{(n_1,n_2)}&= \mathscr Q_{(n_1,n_2-1)}+b^{(2)}{(n_1,n_2-1)}\mathscr Q_{(n_1,n_2)}+c(n_1,n_2)\mathscr Q_{(n_1,n_2+1)}+d^{(2)}(n_1,n_2+1)\mathscr Q_{(n_1+1,n_2+1)}.
\end{align}
For the orthogonal polynomials of type I we have the same relations, cf. \cite[Equation (23.1.20)]{Ismail}: 
\begin{align*}
 xA^{(a)}_{(n_1,n_2)}&= A^{(a)}_{(n_1-1,n_2)}+b^{(1)}{(n_1-1,n_2)}A^{(a)}_{(n_1,n_2)}+c(n_1,n_2)A^{(a)}_{(n_1+1,n_2)}+d^{(1)}(n_1+1,n_2)A^{(a)}_{(n_1+1,n_2+1)},\\
 xA^{(a)}_{(n_1,n_2)}&= A^{(a)}_{(n_1,n_2-1)}+b^{(2)}{(n_1,n_2-1)}A^{(a)}_{(n_1,n_2)}+c(n_1,n_2)A^{(a)}_{(n_1+1,n_2)}+d^{(1)}(n_1+1,n_2)A^{(a)}_{(n_1+1,n_2+1)},\\
 xA^{(a)}_{(n_1,n_2)}&= A^{(a)}_{(n_1-1,n_2)}+b^{(1)}{(n_1-1,n_2)}A^{(a)}_{(n_1,n_2)}+c(n_1,n_2)A^{(a)}_{(n_1,n_2+1)}+d^{(2)}(n_1,n_2+1)A^{(a)}_{(n_1+1,n_2+1)},\\
 xA^{(a)}_{(n_1,n_2)}&= A^{(a)}_{(n_1,n_2-1)}+b^{(2)}{(n_1,n_2-1)}A^{(a)}_{(n_1,n_2)}+c(n_1,n_2)A^{(a)}_{(n_1,n_2+1)}+d^{(2)}(n_1,n_2+1)A^{(a)}_{(n_1+1,n_2+1)},
\end{align*}
for $a\in\{1,2\}$.
In particular,
\begin{align}\label{eq:connection}
 \mathscr Q_{(n_1-1,n_2)}- \mathscr Q_{(n_1,n_2-1)}&=\big(b^{(2)}{(n_1,n_2-1)}-b^{(1)}{(n_1-1,n_2)}\big)\mathscr Q_{(n_1,n_2)},\\
\notag A^{(a)}_{(n_1-1,n_2)}- A^{(a)}_{(n_1,n_2-1)}&=\big(b^{(2)}{(n_1,n_2-1)}-b^{(1)}{(n_1-1,n_2)}\big)A^{(a)}_{(n_1,n_2)}.
\end{align}
\end{teo}
\begin{pro}\label{pro:recurrence_unicity_typeI} 
	Let us assume that the neighbor recursion relations Theorem \ref{Th:Recursion_Relations_I}  hold. Then 
$\mathscr Q_{(n_1,n_2)}$ is  uniquely defined by $\mathscr Q_{(1,0)}$ and $\mathscr Q_{(1,1)}$ and the near neighbor recursion coefficients.
\end{pro}
\begin{proof}
Let us consider the  Equations \eqref{eq:Recursion_Relations_I.2} and \eqref{eq:Recursion_Relations_I.3}  and assume $ (n_1,n_2)=(n,n)$, and $(n_1,n_2) = (n+1,n)$, so that
  \begin{align*}
 x\mathscr Q_{(n,n)}&= \mathscr Q_{(n,n-1)}+b^{(2)}{(n,n-1)}\mathscr Q_{(n,n)}+c(n,n)\mathscr Q_{(n+1,n)}+d^{(1)}(n+1,n)\mathscr Q_{(n+1,n+1)}
\end{align*}
for $n\in\N$ and
\begin{align*}
  x\mathscr Q_{(n+1,n)}&= \mathscr Q_{(n,n)}+b^{(1)}{(n,n)}\mathscr Q_{(n+1,n)}+c(n+1,n)\mathscr Q_{(n+1,n+1)}+d^{(2)}(n+1,n+1)\mathscr Q_{(n+2,n+1)},
\end{align*}
for $ n \in\N_0$, where $Q_{(0,0)} = 0$. From these it follows  that $\mathscr Q_{(n,n)}$ and $ \mathscr Q_{(n+1,n)}$ are uniquely determined by $ \mathscr Q_{(1,0)}$ and  $\mathscr Q_{(1,1)} $ and recursion coefficients.

If we represent in the plane $Oxy$,  the indexes from the step line, i.e. $\big\{(n+1,n+1),(n+1,n)\big\}_{n=0}^\infty$,  belong  respectively to the lines $y-x =0$ and $y-x=-1$. In general, for any index $(n_1, n_2) \not = (0,0)$, such that $n_1, n_2 \geq 0$, there exist $m \in \mathbb{Z}$ such that belongs to $y-x = m$.
The statement of this proposition is equivalent to saying that 
$\mathscr Q_{(n_1,n_2)}$ is defined uniquely by $\mathscr Q_{(1,0)}$ and $\mathscr Q_{(1,1)}$ and recursion coefficients for any index $(n_1, n_2) \not = (0,0)$, such that $n_1, n_2 \geq 0$, belongs to $y-x = m$, for $m \in \mathbb{Z}$. See the figure.

	\begin{center}
	\begin{tikzpicture}
		\draw[Gray!50,thin] (0,0) grid (10,10);
		\foreach \j in {0,...,9}
		\draw[fill=Gray!30] (\j+1,\j+1) circle (3pt);
		\foreach \j in {0,...,9}	\draw[fill=Gray!60] (\j+1,\j) circle (3pt);
		%\foreach \j in {0,...,10} \draw[left] (0,\j )  node {$\j$};
		%\foreach \j in {0,...,10} \draw[below] (\j,0 )  node {$\j$};
		\foreach \j in {1,...,9} \draw[blue,-latex,thick]  (\j,\j)--(\j+1,\j);
		\foreach \j in {0,...,9} \draw[red,-latex,thick]  (\j+1,\j)--(\j+1,\j+1);
		\draw[fill=Black] (9,4) circle (5pt) ;
		\draw[fill=Gray!50] (9,5) circle (5pt) ;	
		\draw[fill=Gray!30] (8,5) circle (5pt) ;
		\draw[->>,thick]  (9,4)--(9,5) ;  \draw[->>,thick] (9,4)--(8,5)  ;
		\draw[fill=Black] (4,9) circle (5pt) ;
		\draw[fill=Gray!30] (5,9) circle (5pt) ;	\draw[fill=Gray!50] (5,8) circle (5pt) ;
		\draw[->>,thick]  (4,9)--(5,9) ;  \draw[->>,thick] (4,9)--(5,8)  ;
		
		\foreach \j in {0,...,6}	\draw[fill=Gray!60] (\j+4,\j) circle (3pt);
		\foreach \j in {0,...,7}	\draw[fill=Gray!30] (\j+3,\j) circle (3pt);
		
		\foreach \j in {0,...,6}	\draw[fill=Gray!30] (\j,\j+4) circle (3pt);
		\foreach \j in {0,...,6}	\draw[fill=Gray!60] (\j,\j+3) circle (3pt);
		\draw[dashed] (0,0)--(10,10);
				\draw[dashed] (1,0)--(10,9);
		\draw[dashed] (3,0)--(10,7);
	\draw[dashed] (4,0)--(10,6);
		\draw[dashed] (0,3)--(7,10);
	\draw[dashed] (0,4)--(6,10);			
		
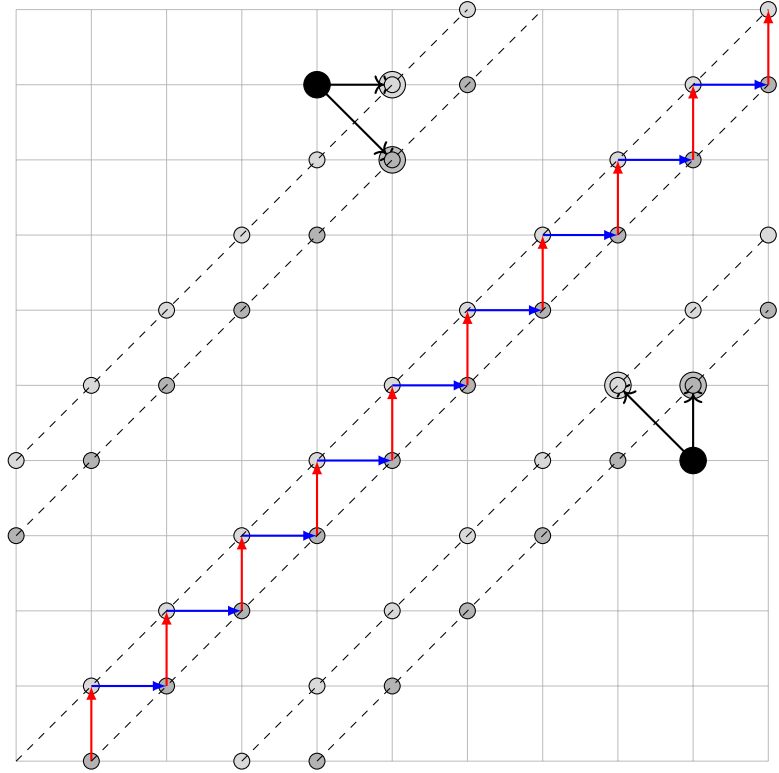
		\draw (5,-.6) node {\textbf { The step line in color and motions to get to the step line}};
		
	\end{tikzpicture}
\end{center}

%\textcolor{red}{In Equation \eqref{eq:connection}, 
%we have  $\big(b^{(2)}{(n_1,n_2-1)}-b^{(1)}{(n_1-1,n_2)}\big) \not =0$, because otherwise $\mathscr Q_{(n_1-1,n_2)}= \mathscr Q_{(n_1,n_2-1)}$ and this is not possible comparing the degrees of the polynomials involved in each of the linear forms. }

Let us prove that for any index $(n_1, n_2) \in\N_0^2$, not the origin, $(n_1, n_2)  \not = (0,0)$,   that belongs to the line $y-x = -n$, for $n \in\N_0$, 
$\mathscr Q_{(n_1,n_2)}$ is defined uniquely by $\mathscr Q_{(1,0)}$ and $\mathscr Q_{(1,1)}$.  We will prove it by induction.
For $n = 0,1$, the step-line,  it is already  proven. Let us assume that it is proven for $n=k$ and prove it for $n=k+1$. 
For this aim we consider $\mathscr Q_{(n_1,n_2)}$ such that  $n_2 -n_1 =-(k+1)$. Using Equation \eqref{eq:connection}, the linear form $ \mathscr Q_{(n_1,n_2)} $ can be written in terms of  the linear forms $\mathscr Q_{(n_1, n_2+1)}$ and $\mathscr Q_{(n_1-1, n_2+1)}$ and recursion coefficients, and the corresponding indexes $(x,y)=(n_1, n_2+1),(n_1-1, n_2+1)$  satisfy  $y-x= ( n_2+1)-n_1= -k$ and
$y-x = (n_2+1)- (n_1-1) = -(k+1)+2 = -(k-1)$, respectively, and by the induction hypotheses we conclude that $\mathscr Q_{(n_1, n_2+1)}$ and $\mathscr Q_{(n_1-1, n_2+1)}$ are written in terms of $ \mathscr Q_{(1,0)}$ and  $\mathscr Q_{(1,1)} $ and, consequently, $ \mathscr Q_{(n_1,n_2)} $ is written in terms of $ \mathscr Q_{(1,0)}$ and  $\mathscr Q_{(1,1)} $ and  recursion coefficients.

Analogously, let us  prove that for any index $(n_1, n_2) \in\N_0^2$, not the origin, $(n_1, n_2)  \not = (0,0)$,   that belongs to the line  $y-x = n-1$, for $n \in\N$, the linear form $\mathscr Q_{(n_1,n_2)}$ is defined uniquely in terms of  $\mathscr Q_{(1,0)}$ and $\mathscr Q_{(1,1)}$ and recursion coefficients. Again we use induction. As it is already  proved for $n = 1$,  let us assume that it holds for $n=k$ and prove it for $n=k+1$. 
Then, if the linear form $\mathscr Q_{(n_1,n_2)}$ such that $n_2 -n_1 =(k+1) -1 = k$, using Equation \eqref{eq:connection}
we can write $\mathscr Q_{(n_1,n_2)}$ in terms of the linear forms  $\mathscr Q_{(n_1+1, n_2)}$ and $\mathscr Q_{(n_1+1, n_2-1)}$, and the corresponding indexes  $ (x,y)=(n_1+1, n_2),(n_1+1, n_2-1)$ satisfy  $y-x = n_2 - (n_1+1 ) = k-1  $.

The figure above illustrates this procedure graphically. From the two previous  paragraphs we see that  the Proposition holds. 
\end{proof}

\section{Hahn multiple orthogonal polynomials of type I }

For $a\in\{1,2\}$ and $N\in\N_0$, let us consider the following weight functions
\begin{align}
 \label{pesosHahn}
 \omega_a(x,\alpha_a,\beta,N)=\dfrac{\Gamma(\alpha_a+x+1)}{\Gamma(\alpha_a+1)\Gamma(x+1)}\frac{\Gamma(\beta+N-x+1)}{\Gamma(\beta+1)\Gamma(N-x+1)},%;\quad i=1,2
\end{align}
defined for $x\in\{0,\dots,N\}$; with and $\alpha_1,\alpha_2,\beta>-1$. It has been proven, cf. \cite{Arvesu}, that the corresponding system of measures
\begin{align}
\label{medidasHahn}
\mu_a&\coloneq \sum_{k=0}^N\omega_a(k,\alpha_i,\beta,N)\delta(x-k),& a&\in\{1,2\},
\end{align}
is perfect. This implies the existence of both type II and I multiple orthogonal polynomials associated to these weight functions and $(n_1,n_2)\in\N_0^2$, which are known as the Hahn multiple orthogonal polynomials. Type II Hahn multiple orthogonal polynomials polynomials $Q_{(n_1,n_2)}(x,\alpha_1,\alpha_2,\beta,N)$ satisfy
the discrete orthogonality relations
\begin{align}
 \label{ortogonalidadHahnII}
 \sum_{k=0}^{N}Q_{(n_1,n_2)}(k,\alpha_1,\alpha_2,\beta,N)\omega_a(k,\alpha_a\beta,N)(-k)_j&=0,&j&=\{0,\dots,n_a-1 \}, & a&\in\{1,2\}, 
\end{align}
with respect to the weight functions \eqref{pesosHahn}.
In this paper we use a frequent notation, cf. \cite[Chapter 6]{Ismail}, to denote the Han polynomials by  $Q_n$. Notice, however that in \cite{Nikiforov_Suslov_Uvarov} they are denoted by  $h_n$. This is why we use a calligraphic letter for the linear form $\mathscr Q_n$, notation that is also usual, see \cite[Chapter 23]{Ismail}.

For $(n_1,n_2)\in\N_0^2$, the explicit expression of these monic type II multiple orthogonal polynomials was found by Arvesú, Coussement and Van Assche, cf. \cite{Arvesu}, and reads
\begin{align}
 \label{HahnTipoII}
\begin{aligned}
  Q_{(n_1,n_2)}&=\begin{multlined}[t][.85\textwidth]
  \dfrac{(\alpha_1+1)_{n_1}(\alpha_2+1)_{n_2}(-N)_{n_1+n_2}}{(\alpha_1+\beta+n_1+n_2+1)_{n_1}(\alpha_2+\beta+n_1+n_2+1)_{n_2}}\\
 \times
 \KF{2:3;1}{2:2;0}{(-x,\alpha_1+\beta+n_1+1):(-n_2,\alpha_1+n_1+1,\alpha_2+\beta+n_1+n_2+1);(-n_1)}{(-N,\alpha_1+1):(\alpha_2+1,\alpha_1+\beta+n_1+1);--}{1,1}
\end{multlined}\\
&=\begin{multlined}[t][.75\textwidth]
  \dfrac{(\alpha_1+1)_{n_1}(\alpha_2+1)_{n_2}(-N)_{n_1+n_2}}{(\alpha_1+\beta+n_1+n_2+1)_{n_1}(\alpha_2+\beta+n_1+n_2+1)_{n_2}}\\
 \times\sum_{l=0}^{n_1+n_2}
 \dfrac{(\alpha_1+\beta+n_1+1)_l}{(\alpha_1+1)_l}\sum_{m=0}^l\dfrac{(-n_1)_{m}(-n_2)_{l-m}}{m!(l-m)!}\dfrac{(\alpha_2+\beta+n_1+n_2+1)_{l-m}(\alpha_1+n_1+1)_{l-m}}{(\alpha_2+1)_{l-m}(\alpha_1+\beta+n_1+1)_{l-m}}\dfrac{(-x)_l}{(-N)_l}.
 \end{multlined}
\end{aligned}
\end{align}
 However, analogous explicit expressions type I multiple orthogonal Hahn polynomials remain unknown till now. The main result of this paper 
 remedies this situation as follows:

\begin{teo}[Type I Hahn multiple orthogonal polynomials]
The type I multiple orthogonal Hahn polynomials $ Q_{(n_1,n_2)}^{(a)}(x,\alpha_1,\alpha_2,\beta,N)$, $a\in\{1,2\}$, $(n_1,n_2)\in\N_0^2$, biorthogonal to the Hahn multiple polynomials of type II in \eqref{HahnTipoII}, are 
\begin{align}
\label{HahnTipoI}
\left\{\begin{aligned}
  Q_{(n_1,n_2)}^{(a)}
 & =
\begin{multlined}[t][.9\textwidth]
  \dfrac{(-1)^{n_a-1}(N+1-n_1-n_2)!(n_1+n_2-2)!}{(n_1-1)!(n_2-1)!(\beta+1)_{n_1+n_2-1}(\alpha_a+\beta+n_1+n_2+n_a)_{N+1-n_1-n_2}}\dfrac{(\hat{\alpha}_a+\beta+\hat{n}_a+1)_{n_1+n_2-1}}{(\alpha_a-\hat{\alpha}_i-\hat{n}_a+1)_{n_1+n_2-1}}\\
 \times\KF{2:3;1}{2:2;0}{(-n_a+1,-N): (\alpha_a+\beta+n_1+n_2, \alpha_a-\hat{\alpha}_a-\hat{n}_a+1,-x); (\hat{\alpha}_a-\alpha_a-n_a+1)}{(-n_1-n_2+2,\hat{\alpha}_a+\beta+\hat{n}_a+1): (\alpha_a+1, -N); --}{1,1}\end{multlined}\\
 & =\sum_{l=0}^{n_a-1}C^{(a),l}_{(n_1,n_2)}(\alpha_1,\alpha_2,\beta,N)(-x)_l,
\end{aligned}\right.
\end{align}
with
\begin{multline}
\label{coefHahnI}
 C^{(a),l}_{(n_1,n_2)}\coloneq\dfrac{(-1)^{n_a-1}(N+1-n_1-n_2)!(n_1+n_2-2)!}{(n_1-1)!(n_2-1)!(\beta+1)_{n_1+n_2-1}(\alpha_a+\beta+n_1+n_2+n_a)_{N+1-n_1-n_2}}\dfrac{(\hat{\alpha}_a+\beta+\hat{n}_a+1)_{n_1+n_2-1}}{(\alpha_a-\hat{\alpha}_a-\hat{n}_a+1)_{n_1+n_2-1}}\\
 \times\dfrac{(\alpha_a+\beta+n_1+n_2)_l
 (\alpha_a-\hat{\alpha}_a-\hat{n}_a+1)_l}{l!(\alpha_a+1)_l(-N)_l}\sum_{m=0}^{n_a-1-l}\dfrac{(-n_a+1)_{l+m}(-N)_{l+m} (\hat{\alpha}_a-\alpha_a-n_a+1)_m}{m!(-n_1-n_2+2)_{l+m}(\hat{\alpha}_a+\beta+\hat{n}_a+1)_{l+m}}.
\end{multline}
We use the notation
\begin{align*}
 \hat{\alpha}_a&\coloneq \delta_{a,2}\alpha_1+\delta_{a,1}\alpha_2, &\hat{n}_a&\coloneq \delta_{a,2}n_1+\delta_{a,1}n_2=n_1+n_2-n_a.
\end{align*}
\end{teo}\enlargethispage{1cm}
To better understand the notation notice that $\hat{\alpha}_1=\alpha_2,\hat{\alpha}_2=\alpha_1$ and $\hat{n}_1=n_2,\hat{n}_2=n_1$ and, consequently,
\begin{align*}
 Q_{(n_1,n_2)}^{(1)}(x)
 &=\begin{multlined}[t][.7\textwidth]
  \dfrac{(-1)^{n_1-1}(N+1-n_1-n_2)!(n_1+n_2-2)!}{(n_1-1)!(n_2-1)!(\beta+1)_{n_1+n_2-1}(\alpha_1+\beta+n_1+n_2+n_1)_{N+1-n_1-n_2}}\dfrac{({\alpha}_2+\beta+{n}_2+1)_{n_1+n_2-1}}{(\alpha_1-{\alpha}_2-{n}_2+1)_{n_1+n_2-1}}\\
 \times\KF{2:3;1}{2:2;0}{(-n_1+1,-N): (\alpha_1+\beta+n_1+n_2, \alpha_1-{\alpha}_2-{n}_2+1,-x); ({\alpha}_2-\alpha_1-n_1+1)}{(-n_1-n_2+2,{\alpha}_2+\beta+{n}_2+1): (\alpha_1+1, -N); --}{1,1},
 \end{multlined}\\
 Q_{(n_1,n_2)}^{(2)}(x)
 &=\begin{multlined}[t][.7\textwidth]\dfrac{(-1)^{n_2-1}(N+1-n_1-n_2)!(n_1+n_2-2)!}{(n_1-1)!(n_2-1)!(\beta+1)_{n_1+n_2-1}(\alpha_2+\beta+n_1+n_2+n_2)_{N+1-n_1-n_2}}\dfrac{({\alpha}_1+\beta+{n}_1+1)_{n_1+n_2-1}}{(\alpha_2-{\alpha}_1-{n}_1+1)_{n_1+n_2-1}}\nonumber\\
 \times\KF{2:3;1}{2:2;0}{(-n_2+1,-N): (\alpha_2+\beta+n_1+n_2, \alpha_2-{\alpha}_1-{n}_1+1,-x); ({\alpha}_1-\alpha_2-n_2+1)}{(-n_1-n_2+2,{\alpha}_1+\beta+{n}_1+1): (\alpha_2+1, -N); --}{1,1}.
 \end{multlined}
\end{align*}

In order to prove this Theorem, we will first prove the orthogonality relations and then prove the biorthogonality, finally we will state the recursion relations satisfied by these multiple orthogonal polynomials.
%%%%%%%%%%%%%%%%%%%

\subsection{Orthogonality}
We are going to start proving the polynomials defined in \eqref{HahnTipoI} satisfy discrete orthogonality relations of the form \eqref{ortdiscI}. In order to do that we will need the following reduction formula for Kampé de Fériet series to a generalized hypergeometric series.

\begin{lemma}
\label{conjeturaqueyano}
Let be $n_a,\hat{n}_a\in\mathbb N$, $\alpha_a,\hat{\alpha}_a,\beta,N,j\in\mathbb R$ then
\begin{multline}
\label{laquehayqueprobar}
\KF{2:3;1}{2:2;0}{(-n_a+1,-N): (\alpha_a+\beta+n_1+n_2,-N+j, \alpha_a-\hat{\alpha}_a-\hat{n}_a+1); (\hat{\alpha}_a-\alpha_a-n_a+1)}{(-n_1-n_2+2,\hat{\alpha}_a+\beta+\hat{n}_a+1): (\alpha_a+\beta+j+2, -N); --}{1,1}\\
 =\dfrac{(\hat{\alpha}_a-\alpha_a-n_a+1)_{n_a-1}(\alpha_a+\beta+N+2)_{n_a-1}}{(-n_1-n_2+2)_{n_a-1}(\hat{\alpha}_a+\beta+\hat{n}_a+1)_{n_a-1}}\pFq{3}{2}{-n_a+1,\alpha_a+\beta+n_1+n_2,\alpha_a-\hat{\alpha}_a-\hat{n}_a+1}{\alpha_a+\beta+j+2,\alpha_a-\hat{\alpha}_a+1}{1}.
\end{multline}
\end{lemma}

\begin{proof}
Let's write the left hand side of Equation \eqref{laquehayqueprobar} as
\begin{multline}
\label{miembroizquierdo}
 \KF{2:3;1}{{2:2;0}}{(-n_a+1,-N): (\alpha_a+\beta+n_1+n_2,-N+j, \alpha_a-\hat{\alpha}_a-\hat{n}_a+1); (\hat{\alpha}_a-\alpha_a-n_a+1)}{(-n_1-n_2+2,\hat{\alpha}_a+\beta+\hat{n}_a+1): (\alpha_a+\beta+j+2, -N); --}{1,1}
 \\=\sum^{n_a-1}_{l=0}\dfrac{1}{l!}\dfrac{(-n_a+1)_l}{(-n_1-n_2+2)_l(\hat{\alpha}_a+\beta+\hat{n}_a+1)_l}\dfrac{(\alpha_a+\beta+n_1+n_2)_l(-N+j)_l(\alpha_a-\hat{\alpha}_a-\hat{n}_a+1)_l}{(\alpha_a+\beta+j+2)_l}\\
 \times\sum^{n_a-1-l}_{m=0}
 \dfrac{1}{m!}\dfrac{(-n_a+1+l)_{m}(-N+l)_m (\hat{\alpha}_a-\alpha_a-n_a+1)_m}{(-n_1-n_2+2+l)_{m}(\hat{\alpha}_a+\beta+\hat{n}_a+1+l)_{m}},
\end{multline}
and recall the Chu--Vandermonde identity to get
\begin{align}
\label{cocienteconVanderchu}
 &\dfrac{(-N+j)_l}{(\alpha_a+\beta+j+2)_l}=\sum_{k=0}^l\dfrac{(-l)_k}{k!}\dfrac{(\alpha_a+\beta+N+2)_k}{(\alpha_a+\beta+j+2)_k}.
\end{align}
Then, if we introduce \eqref{cocienteconVanderchu} into \eqref{miembroizquierdo}, after some reorganization, we obtain
\begin{multline}
\label{miembroizquierdo2}
 \KF{2:3;1}{2:2;0}{(-n_a+1,-N): (\alpha_a+\beta+n_1+n_2,-N+j, \alpha_a-\hat{\alpha}_a-\hat{n}_a+1); (\hat{\alpha}_a-\alpha_a-n_a+1)}{(-n_1-n_2+2,\hat{\alpha}_a+\beta+\hat{n}_a+1): (\alpha_a+\beta+j+2, -N); --}{1,1}\\
 =\sum^{n_a-1}_{k=0}\dfrac{1}{k!}\dfrac{(-n_a+1)_{k}(\alpha_a+\beta+n_1+n_2)_{k}(\alpha_a-\hat{\alpha}_a-\hat{n}_a+1)_{k}}{(\alpha_a+\beta+j+2)_k(\alpha_a-\hat{\alpha}_a+1)_k}
 (-1)^k\dfrac{(\alpha_a+\beta+N+2)_k(\alpha_a-\hat{\alpha}_a+1)_k}{(-n_1-n_2+2)_{k}(\hat{\alpha}_a+\beta+\hat{n}_a+1)_{k}}\\
 \times\KF{2:2;1}{2:1;0}{(-n_a+1+k,-N+k):(\alpha_a+\beta+n_1+n_2+k,{\alpha}_a-\hat{\alpha}_a-\hat{n}_a+1+k);(\hat{\alpha}_a-\alpha_a-n_a+1)}{(-n_1-n_2+2+k,\hat{\alpha}_a+\beta+\hat{n}_a+1+k):(-N+k);--}{1,1}.
\end{multline}
Now, we use the following reduction formula due to Rakha and Rathie, cf. \cite[Equation (3.20)]{RR},
\begin{align*}
%\label{RR}
\KF{2:2;1}{2:1;0}{\alpha,\lambda:\gamma,\beta-\epsilon;\epsilon}{\beta,\mu:\delta;--}{1,1}=\dfrac{\Gamma(\mu)\Gamma(\mu-\alpha-\lambda)}{\Gamma(\mu-\alpha)\Gamma(\mu-\lambda)}\,\pFq{4}{3}{\alpha,\lambda,\delta-\gamma,\beta-\epsilon}{\beta,\delta,1-\mu+\alpha+\lambda}{1}
\end{align*}
to rewrite
\begin{multline*}
	%\label{hypergeomconRR}
	\KF{2:2;1}{2:1;0}{(-n_i+1+k,-N+k):(\alpha_i+\beta+n_1+n_2+k,{\alpha}_i-\hat{\alpha}_i-\hat{n}_i+1+k);(\hat{\alpha}_i-\alpha_i-n_i+1)}{(-n_1-n_2+2+k,\hat{\alpha}_i+\beta+\hat{n}_i+1+k):(-N+k);--}{1,1}\\
	=\begin{multlined}[t][0.85\textwidth]
		\dfrac{\Gamma(\hat{\alpha}_i+\beta+\hat{n}_i+1+k)\Gamma(\hat{\alpha}_i+\beta+n_1+n_2+N-k)}{\Gamma(\hat{\alpha}_i+\beta+{n}_1+n_2)\Gamma(\hat{\alpha}_i+\beta+\hat{n}_i+N+1)}\\
	\times\pFq{4}{3}{-n_i+1+k,-N+k,\alpha_i-\hat{\alpha}_i-\hat{n}_i+1+k,-\alpha_i-\beta-n_1-n_2-N}{-n_1-n_2+2+k,-N+k,-\hat{\alpha}_i-\beta-n_1-n_2-N+k+1}{1}
	\end{multlined}\\
=\begin{multlined}[t][.75\textwidth]
	\dfrac{\Gamma(\hat{\alpha}_i+\beta+\hat{n}_i+1+k)\Gamma(\hat{\alpha}_i+\beta+n_1+n_2+N-k)}{\Gamma(\hat{\alpha}_i+\beta+{n}_1+n_2)\Gamma(\hat{\alpha}_i+\beta+\hat{n}_i+N+1)}\nonumber\\
	\times\pFq{3}{2}{-n_i+1+k,\alpha_i-\hat{\alpha}_i-\hat{n}_i+1+k,-\alpha_i-\beta-n_1-n_2-N}{-n_1-n_2+2+k,-\hat{\alpha}_i-\beta-n_1-n_2-N+k+1}{1}.
\end{multlined}
\end{multline*}
%\begin{multline*}
%%\label{hypergeomconRR}
% \KF{2:2;1}{2:1;0}{(-n_a+1+k,-N+k):(\alpha_a+\beta+n_1+n_2+k,{\alpha}_a-\hat{\alpha}_a-\hat{n}_a+1+k);(\hat{\alpha}_a-\alpha_a-n_a+1)}{(-n_1-n_2+2+k,\hat{\alpha}_a+\beta+\hat{n}_a+1+k):(-N+k);--}{1,1}\\
% =\dfrac{\Gamma(\hat{\alpha}_a+\beta+\hat{n}_a+1+k)\Gamma(\hat{\alpha}_a+\beta+n_1+n_2+N-k)}{\Gamma(\hat{\alpha}_a+\beta+{n}_1+n_2)\Gamma(\hat{\alpha}_a+\beta+\hat{n}_a+N+1)}\\
% \times\pFq{3}{2}{-n_a+1+k,\alpha_a-\hat{\alpha}_a-\hat{n}_a+1+k,-\alpha_a-\beta-n_1-n_2-N}{-n_1-n_2+2+k,-\hat{\alpha}_a-\beta-n_1-n_2-N+k+1}{1}.
%\end{multline*}
Then, the Pfaff--Saalschütz formula, cf. \cite[Theorem 2.2.6]{LibrodeHypergeom}:
\begin{align*}
%\label{hypergeom32}
 \pFq{3}{2}{-n,a,b}{c,d}{1}&=\dfrac{(c-a)_{n}(c-b)_{n}}{(c)_n{}(c-a-b)_{n}},& &\text{if $-n+a+b+1=c+d$,}
\end{align*}
when used in the previous equation leads to the following summation formula at unity for a Kampé de Fériet series
\begin{multline}
 \label{hypergeomconladela32}
  \KF{2:2;1}{2:1;0}{(-n_a+1+k,-N+k):(\alpha_a+\beta+n_1+n_2+k,{\alpha}_a-\hat{\alpha}_a-\hat{n}_a+1+k);(\hat{\alpha}_a-\alpha_a-n_a+1)}{(-n_1-n_2+2+k,\hat{\alpha}_a+\beta+\hat{n}_a+1+k):(-N+k);--}{1,1}\\
 =\dfrac{\Gamma(\hat{\alpha}_a+\beta+\hat{n}_a+1+k)\Gamma(\hat{\alpha}_a+\beta+n_1+n_2+N-k)}{\Gamma(\hat{\alpha}_a+\beta+{n}_1+n_2)\Gamma(\hat{\alpha}_a+\beta+\hat{n}_a+N+1)}
 \dfrac{(\hat{\alpha}_a-\alpha_a-n_a+1)_{n_a-1-k}(\alpha_a+\beta+N+2+k)_{n_a-1-k}}{(-n_1-n_2+2+k)_{n_a-1-k}(\hat{\alpha}_a+\beta+\hat{n}_a+1+N)_{n_a-1-k}}.\end{multline}
If we introduce \eqref{hypergeomconladela32} 
into Equation \eqref{miembroizquierdo2}, and do some clearing, we finally find that
\begin{multline*}
\KF{2:3;1}{2:2;0}{(-n_a+1,-N): (\alpha_a+\beta+n_1+n_2,-N+j, \alpha_a-\hat{\alpha}_a-\hat{n}_a+1); (\hat{\alpha}_a-\alpha_a-n_a+1)}{(-n_1-n_2+2,\hat{\alpha}_a+\beta+\hat{n}_a+1): (\alpha_a+\beta+j+2, -N); --}{1,1}\\
 =\dfrac{(\hat{\alpha}_a-\alpha_a-n_a+1)_{n_a-1}(\alpha_a+\beta+N+2)_{n_a-1}}{(-n_1-n_2+2)_{n_a-1}(\hat{\alpha}_a+\beta+\hat{n}_a+1)_{n_a-1}}\pFq{3}{2}{-n_a+1,\alpha_a+\beta+n_1+n_2,\alpha_a-\hat{\alpha}_a-\hat{n}_a+1}{\alpha_a+\beta+j+2,\alpha_a-\hat{\alpha}_a+1}{1}.
\end{multline*}

\end{proof}

With Lemma \ref{conjeturaqueyano} and Karp--Prilepkina's Theorem \ref{Karp} we are able to prove the multiple orthogonality relations.
%%%%%%%%%%%%%%%%%%%%%%%%%%%%%%%%%%%%%%%%%%%%%%%%%%%%%%%%
%%%%%%%%%%%%%%%%%%%%%%%%%%%%%%%%%%%%%%%%%%%%%%%%%%%%%%%%
\begin{pro}
\label{laprop}
For $ j\in \{0,\dots,n_1+n_2-2\}$, the polynomials \eqref{HahnTipoI} satisfy the orthogonality relations
\begin{align}
\label{ortogonalidadHahnInofinal}
 \sum^{N}_{k=0}\left(Q_{(n_1,n_2)}^{(1)}(k,\alpha_1,\alpha_2,\beta,N)\omega_1(k,\alpha_1,\beta,N)+Q_{(n_1,n_2)}^{(2)}(k,\alpha_1,\alpha_2,\beta,N) \omega_2(k,\alpha_2,\beta,N)\right)(-N+k)_j&=0,\\
\label{ortogonalidadHahnIfinal}
 \sum^{N}_{k=0}\left(Q_{(n_1,n_2)}^{(1)}(k,\alpha_1,\alpha_2,\beta,N)\omega_1(k,\alpha_1,\beta,N)
 +Q_{(n_1,n_2)}^{(2)}(k,\alpha_1,\alpha_2,\beta,N) \omega_2(k,\alpha_2,\beta,N)\right)(-N+k)_{n_1+n_2-1}&=1,
\end{align}
with respect to the weight functions \eqref{pesosHahn} .
\end{pro}

\begin{proof}
First of all we introduce the convenient notation 
\begin{align*}
%\label{integralHahn}
 I^{(a),j}_{(n_1,n_2)}(\alpha_1,\alpha_2,\beta,N)\coloneq\sum^{N}_{k=0}Q_{(n_1,n_2)}^{(a)}(k,\alpha_1,\alpha_2,\beta,N)\omega_{a}(k,\alpha_a,\beta,N)(-N+k)_j
\end{align*}
to work with.
Equation  \eqref{pesosHahn} for $\omega_a(k,\alpha_a,\beta,N)$, Equation \eqref{HahnTipoI} for $Q_{(n_1,n_2),i}(k,\alpha_1,\alpha_2,\beta,N)$ and some clearing
leads to
\begin{align*}
%\label{integralHahn2}
%\hspace*{-1cm}\begin{aligned}
  I^{(a),j}_{(n_1,n_2)}&\begin{aligned}[t]
&=\sum^{n_a-1}_{l=0}C^{(a),l}_{(n_1,n_2)}(\alpha_1,\alpha_2,\beta,N)\sum^{N}_{k=0}\dfrac{\Gamma(\alpha_a+k+1)}{\Gamma(\alpha_a+1)}\dfrac{\Gamma(\beta+N-k+1)}{\Gamma(\beta+1)}\dfrac{(-k)_l}{\Gamma(k+1)}\dfrac{(-N+k)_j}{\Gamma(N-k+1)}\\
&=\sum^{n_a-1}_{l=0}(-1)^{j+l}\dfrac{(\alpha_a+1)_l(\beta+1)_j}{(N-j-l)!}C^{(a),l}_{(n_1,n_2)}(\alpha_1,\alpha_2,\beta,N)\sum^{N-j}_{k=l}\dbinom{N-j-l}{k-l}(\alpha_a+l+1)_{k-l}(\beta+j+1)_{N-j-k}\\
&=\sum^{n_a-1}_{l=0}(-1)^{j+l}\dfrac{(\alpha_a+1)_l(\beta+1)_j}{(N-j-l)!}C^{(a),l}_{(n_1,n_2)}(\alpha_1,\alpha_2,\beta,N)(\alpha_a+\beta+j+l+2)_{N-j-l},
\end{aligned}
%\end{aligned}
\end{align*}
where the $k$-labeled sum can be easily simplified by means of the Newton's binomial formula for Pochhammer's symbols (which follows from Chu--Vandermonde relation):
\begin{align*}
 % \label{Vanderchu}
 (a+b)_n=\sum_{k=0}^n\dbinom{n}{k}(a)_k(b)_{n-k}.
\end{align*}

Using Equation \eqref{coefHahnI} for $C^{(a),l}_{(n_1,n_2),i}$ and after some clearing, we find
\begin{align*}
%\label{integralHahn3}
I^{(a),j}_{(n_1,n_2)}
=\begin{multlined}[t][.9\textwidth]
 \dfrac{(-1)^{n_a-1+j}(n_1+n_2-2)!(N-n_1-n_2+1)!}{(n_1-1)!(n_2-1)!(N-j)!(\beta+j+1)_{n_1+n_2-1-j}}\dfrac{(\alpha_a+\beta+j+2)_{n_1+n_2+n_a-2-j}(\hat{\alpha}_a+\beta+\hat{n}_a+1)_{n_1+n_2-1}}{(\alpha_a-\hat{\alpha}_a-\hat{n}_a+1)_{n_1+n_2-1}(\alpha_a+\beta+N+2)_{n_a-1}}\\
\times\KF{2:3;1}{2:2;0}{(-n_a+1,-N): (\alpha_a+\beta+n_1+n_2,-N+j, \alpha_a-\hat{\alpha}_a-\hat{n}_a+1); (\hat{\alpha}_a-\alpha_a-n_a+1)}{(-n_1-n_2+2,\hat{\alpha}_a+\beta+\hat{n}_a+1): (\alpha_a+\beta+j+2, -N); --}{1,1}.
\end{multlined}
\end{align*}
Then, Lemma \ref{conjeturaqueyano} helps in the reduction of the previous Kampé de Fériet series and we get
\begin{align}
\label{integralHahnsimplificada}
I^{(a),j}_{(n_1,n_2)}
&=\begin{multlined}[t][.9\textwidth]\dfrac{(-1)^{n_1+n_2-1-j}(N-n_1-n_2+1)!}{(N-j)!(\beta+j+1)_{n_1+n_2-1-j}}
\dfrac{(\alpha_a+\beta+j+2)_{n_1+n_2+n_a-2-j}(\hat{\alpha}_a+\beta+n_1+n_2)_{\hat{n}_a}}{(n_a-1)!(\hat{\alpha}_a-\alpha_a)_{\hat{n}_a}}\\
\times\pFq{3}{2}{-n_a+1,\alpha_a+\beta+n_1+n_2,\alpha_a-\hat{\alpha}_a-\hat{n}_a+1}{\alpha_a+\beta+j+2,\alpha_a-\hat{\alpha}_a+1}{1}.
\end{multlined}
\end{align}

Once we have arrived to this expression, we are going to distinguish the cases $j\in\{0,\dots,n_1+n_2-2\}$ and $j=n_1+n_2-1$. In the first place, let us prove that $I^{(1),j}_{(n_1,n_2)}+I^{(2),j}_{(n_1,n_2)}=0$, whenever $j<n_1+n_2-1$.
Let's take $a=1$. Then, if $j<n_1+n_2-1$, we can apply Theorem \ref{Karp} to rewrite the previous ${}_3F_2$ series as
\begin{align*}
\hspace*{-.5cm}\begin{aligned}
  \pFq{3}{2}{-n_1+1,\alpha_1+\beta+n_1+n_2,\alpha_1-{\alpha}_2-{n}_2+1}{\alpha_1+\beta+j+2,\alpha_1-{\alpha}_2+1}{1}
 &=\begin{multlined}[t][.6\textwidth]
  \dfrac{({n}_1-1)!}{({n}_2-1)!}
 \dfrac{\Gamma(\alpha_1-{\alpha}_2+1)}{\Gamma(\alpha_1-{\alpha}_2-{n}_2+1+n_1)}
 \dfrac{({\alpha}_2+\beta+j+{n}_2+1)_{n_1+n_2-2-j}}{(\alpha_1+\beta+j+2)_{n_1+n_2-2-j}}\\
 \times\pFq{3}{2}{-{n}_2+1,-{\alpha}_2-\beta-j-{n}_2,\alpha_1-{\alpha}_2-{n}_2+1}{-{\alpha}_2-\beta-{n}_2+2-n_1-n_2,\alpha_1-{\alpha}_2-{n}_2+n_1+1}{1},
\end{multlined}
\end{aligned}
\end{align*}
and the following reversal formula, cf. \cite[Equation 2.2.3.2]{LibrodeKF},
\begin{align*}
%\label{cambiosignohypergeom}
 &\pFq{p+1}{q}{-n,a_1,\dots,a_p}{b_1,\dots,b_q}{1}=(-1)^n\dfrac{(a_1)_n\cdots(a_p)_n}{(b_1)_n\cdots(b_q)_n}\pFq{q+1}{p}{-n,-b_1-n+1,\dots,-b_q-n+1}{-a_1-n+1,\dots,-a_p-n+1}{1}
\end{align*}
lead us to
\begin{multline*}
 \pFq{3}{2}{-n_1+1,\alpha_1+\beta+n_1+n_2,\alpha_1-{\alpha}_2-{n}_2+1}{\alpha_1+\beta+j+2,\alpha_1-{\alpha}_2+1}{1}\\
  =\dfrac{({n}_1-1)!}{({n}_2-1)!}
 \dfrac{\Gamma(\alpha_1-{\alpha}_2+1)}{\Gamma(\alpha_1-{\alpha}_2-{n}_2+1+n_1)}
 \dfrac{({\alpha}_2+\beta+j+{n}_2+1)_{n_1+n_2-2-j}}{(\alpha_1+\beta+j+2)_{n_1+n_2-2-j}}\\
 \times(-1)^{{n}_2-1}
 \dfrac{(-{\alpha}_2-\beta-j-{n}_2)_{{n}_2-1}(\alpha_1-{\alpha}_2-{n}_2+1)_{{n}_2-1}}{(-{\alpha}_2-\beta-{n}_2+2-n_1-n_2)_{{n}_2-1}(\alpha_1-{\alpha}_2-{n}_2+n_1+1)_{{n}_2-1}}\\
 \times\pFq{3}{2}{-{n}_2+1,{\alpha}_2+\beta+n_1+n_2,{\alpha}_2-\alpha_1-n_1+1}{{\alpha}_2+\beta+j+2,{\alpha}_1-\alpha_1+1}{1}.
\end{multline*}
When we replace this expression in Equation \eqref{integralHahnsimplificada}, for $a=1$, we find
\begin{align*}
I^{(1),j}_{(n_1,n_2)}
&=\begin{multlined}[t][.7\textwidth]
 \dfrac{(-1)^{n_1+n_2-1-j}(N-n_1-n_2+1)!}{(N-j)!(\beta+j+1)_{n_1+n_2-1-j}}
\dfrac{(\alpha_1+\beta+j+2)_{n_1+n_2+n_1-2-j}({\alpha}_2+\beta+n_1+n_2)_{{n}_2}}{(n_1-1)!({\alpha}_2-\alpha_1)_{{n}_2}}\\
\times\pFq{3}{2}{-n_1+1,\alpha_1+\beta+n_1+n_2,\alpha_1-{\alpha}_2-{n}_2+1}{\alpha_1+\beta+j+2,\alpha_1-{\alpha}_2+1}{1}
\end{multlined}\\
&=\begin{multlined}[t][.7\textwidth]
 -\dfrac{(-1)^{n_1+n_2-1-j}(N-n_1-n_2+1)!}{(N-j)!(\beta+j+1)_{n_1+n_2-1-j}}
\dfrac{({\alpha}_2+\beta+j+2)_{n_1+n_2+{n}_2-2-j}({\alpha}_1+\beta+n_1+n_2)_{{n}_1}}{({n}_2-1)!({\alpha}_1-{\alpha}_2)_{{n}_1}}\\
\times\pFq{3}{2}{-{n}_2+1,{\alpha}_2+\beta+n_1+n_2,{\alpha}_2-\alpha_1-n_1+1}{{\alpha}_2+\beta+j+2,{\alpha}_2-\alpha_1+1}{1}
\end{multlined}\\
&=-I^{(2),j}_{(n_1,n_2)},
\end{align*}
and we have proven the orthogonality relations \eqref{ortogonalidadHahnInofinal}.
%%%%%%%%%%%%%%%%%%%%%%%%%%%%%%%%%%%%%%%%%%%%%%%%%%%%%%%%%%
Finally, let us discuss the case $j=n_1+n_2-1$, so that we get unity normalization. We start from Equation \eqref{integralHahnsimplificada}. First, we rewrite $I^{(a),n_1+n_2-1}_{(n_1,n_2)}$ as
\begin{align}
\label{integralHahncasofinal}
\begin{aligned}
 I^{(a),n_1+n_2-1}_{(n_1,n_2)}
&=\dfrac{(\alpha_a+\beta+n_1+n_2+1)_{n_a-1}(\hat{\alpha}_a+\beta+n_1+n_2)_{\hat{n}_a}}{(n_a-1)!(\hat{\alpha}_a-\alpha_a)_{\hat{n}_a}}
\pFq{3}{2}{-n_a+1,\alpha_a+\beta+n_1+n_2,\alpha_a-\hat{\alpha}_a-\hat{n}_a+1}{\alpha_a+\beta+n_1+n_2+1,\alpha_a-\hat{\alpha}_a+1}{1}\\
&=\dfrac{(\alpha_a+\beta+n_1+n_2+1)_{n_a-1}(\hat{\alpha}_a+\beta+n_1+n_2)_{\hat{n}_a}}{(n_a-1)!(\hat{\alpha}_a-\alpha_a)_{\hat{n}_a}}
\sum_{l=0}^{n_a-1}\dfrac{(-n_a+1)_l}{l!}\dfrac{(\alpha_a+\beta+n_1+n_2)_l(\alpha_a-\hat{\alpha}_a-\hat{n}_a+1)_l}{(\alpha_a+\beta+n_1+n_2+1)_l(\alpha_a-\hat{\alpha}_a+1)_l}\\
&=\begin{multlined}[t][.9\textwidth]
 ({\alpha}_1+\beta+n_1+n_2)_{{n}_1}(\alpha_2+\beta+n_1+n_2)_{n_2}\\
\times\dfrac{1}{(n_a-1)!}\sum^{n_a-1}_{l=0}(-1)^{\hat{n}_a+l}\binom{n_a-1}{l}\dfrac{1}{(\alpha_a+l+\beta+n_1+n_2)(\alpha_a+l-(\hat{\alpha}_a+\hat{n}_a-1))_{\hat{n}_a}}.
\end{multlined}
\end{aligned}
\end{align}

Let's remind now the following simple fraction expansion formula
\begin{align*}
%\label{decomp}
 \dfrac{1}{(z-b)(z-a)_n}=\dfrac{1}{(z-b)(b-a)_n}+\dfrac{1}{(n-1)!}\sum_{p=0}^{n-1}(-1)^p\dbinom{n-1}{p}\dfrac{1}{z-a+p},
\end{align*}
so, taking $\alpha_i+l$ as variable, we can rewrite the fraction within the sum on (\ref{integralHahncasofinal}) as
\begin{multline*}
 \dfrac{1}{(\alpha_a+l+\beta+n_1+n_2)(\alpha_a+l-(\hat{\alpha}_a+\hat{n}_a-1))_{\hat{n}_a}}
 \\\begin{aligned}[t]&=
  \begin{multlined}[t][.7\textwidth]
  \dfrac{1}{(\alpha_a+l+\beta+n_1+n_2)(-\hat{\alpha}_a-\beta-n_1-n_2-\hat{n}_a+1)_{\hat{n}_a}}\\
 +\dfrac{1}{(\hat{n}_a-1)!}\sum_{p=0}^{\hat{n}_a-1}(-1)^p\dbinom{\hat{n}_a-1}{p}\dfrac{1}{(\hat{\alpha}_a+\beta+\hat{n}_a-1-p+n_1+n_2)(\alpha_a+l-\hat{\alpha}_a-\hat{n}_a+1+p)}
 \end{multlined}\\
 &=\begin{multlined}[t][.7\textwidth]
  \dfrac{(-1)^{\hat{n}_a}}{(\alpha_a+l+\beta+n_1+n_2)(\hat{\alpha}_a+\beta+n_1+n_2)_{\hat{n}_a}}\\
 +\dfrac{1}{(\hat{n}_a-1)!}\sum_{p=0}^{\hat{n}_a-1}(-1)^p\dbinom{\hat{n}_a-1}{p}\dfrac{1}{(\hat{\alpha}_a+\beta+\hat{n}_a-1-p+n_1+n_2)(\alpha_a+l-\hat{\alpha}_a-\hat{n}_a+1+p)}.
 \end{multlined}
 \end{aligned}
\end{multline*}
Replacing this expression in Equation \eqref{integralHahncasofinal} for $a=1$ we find that
\begin{align*}
 I^{(1),n_1+n_2-1}_{(n_1,n_2),1}
 &=\begin{multlined}[t][.75\textwidth]
  ({\alpha}_1+\beta+n_1+n_2)_{{n}_1}(\alpha_2+\beta+n_1+n_2)_{n_2}\\
 \times \dfrac{1}{(n_1-1)!}\sum^{n_1-1}_{l=0}(-1)^{{n}_2+l}\binom{n_1-1}{l}\dfrac{1}{(\alpha_1+l+\beta+n_1+n_2)(\alpha_1+l-({\alpha}_2+{n}_2-1))_{{n}_2}}
 \end{multlined}\\
 &=\begin{multlined}[t][.75\textwidth]
  \dfrac{({\alpha}_1+\beta+n_1+n_2)_{{n}_1}(\alpha_2+\beta+n_1+n_2)_{n_2}}{({\alpha}_2+\beta+n_1+n_2)_{{n}_2}}\dfrac{1}{(n_1-1)!}\sum^{n_1-1}_{l=0}(-1)^{l}\binom{n_1-1}{l}\dfrac{1}{\alpha_1+\beta+n_1+n_2+l}\\
 -\dfrac{({\alpha}_1+\beta+n_1+n_2)_{{n}_1}(\alpha_2+\beta+n_1+n_2)_{n_2}}{({n}_2-1)!}\sum_{p=0}^{{n}_2-1}(-1)^p\dbinom{{n}_2-1}{p}\dfrac{1}{({\alpha}_2+\beta+{n}_2-1-p+n_1+n_2)}\\
 \times\dfrac{1}{(n_1-1)!}\sum^{n_1-1}_{l=0}(-1)^{{n}_2+l}\binom{n_1-1}{l}\dfrac{1}{{\alpha}_2+{n}_2-1-p-\alpha_1-l}.
 \end{multlined}
\end{align*}
Doing the index interchanges $p\rightarrow n_2-1-p$ and $l\rightarrow n_1-1-l$ on the second summation we find that
\begin{align*}
 I^{(1),n_1+n_2-1}_{(n_1,n_2)}
 &=\begin{multlined}[t][.7\textwidth]
  \dfrac{({\alpha}_1+\beta+n_1+n_2)_{{n}_1}(\alpha_2+\beta+n_1+n_2)_{n_2}}{({\alpha}_2+\beta+n_1+n_2)_{{n}_2}}\dfrac{1}{(n_1-1)!}\sum^{n_1-1}_{l=0}(-1)^{l}\binom{n_1-1}{l}\dfrac{1}{\alpha_1+\beta+n_1+n_2+l}\\
 -\dfrac{({\alpha}_1+\beta+n_1+n_2)_{{n}_1}(\alpha_2+\beta+n_1+n_2)_{n_2}}{({n}_2-1)!}\sum^{{n}_2-1}_{p=0}(-1)^{n_1+p}\dbinom{{n}_2-1}{p}\dfrac{1}{({\alpha}_2+\beta+n_1+n_2+p)}\\
 \times\dfrac{1}{(n_1-1)!}\sum^{n_1-1}_{l=0}(-1)^{l}\binom{n_1-1}{l}\dfrac{1}{{\alpha}_2+p-(\alpha_1+n_1-1)+l}.
 \end{multlined}
\end{align*}
Finally, we use that
\begin{align*}
%\label{decomp2}
 \dfrac{1}{(n-1)!}\sum_{l=0}^{n-1}(-1)^l\dbinom{n-1}{l}\dfrac{1}{z-a+l}=\dfrac{1}{(z-a)_n},
\end{align*}
to deal with the last summation and get that
\begin{align*}
 I^{(1),n_1+n_2-1}_{(n_1,n_2)}
 &=\begin{multlined}[t][.7\textwidth]
  \dfrac{({\alpha}_1+\beta+n_1+n_2)_{{n}_1}(\alpha_2+\beta+n_1+n_2)_{n_2}}{({\alpha}_2+\beta+n_1+n_2)_{{n}_2}(\alpha_1+\beta+n_1+n_2)_{n_1}}
  -\dfrac{({\alpha}_1+\beta+n_1+n_2)_{{n}_1}(\alpha_2+\beta+n_1+n_2)_{n_2}}{({n}_2-1)!}\\
 \times\sum^{{n}_2-1}_{p=0}(-1)^{n_1+p}\dbinom{{n}_2-1}{p}\dfrac{1}{({\alpha}_2+\beta+n_1+n_2+p)({\alpha}_2+p-(\alpha_1+n_1-1))_{n_1}}
 \end{multlined}\\
 &=1-I^{(2),n_1+n_2-1}_{(n_1,n_2)},
\end{align*}
and \eqref{ortogonalidadHahnIfinal} is proven.
\end{proof}

%%%%%%%%%%%%%%%%%%%
\subsection{Biorthogonality}

%Now we are going to prove that the polynomials defined at (\ref{HahnTipoI}) satisfy biortogonalithy relations of the form (\ref{biortogonalidaddisc}) respect to the Hahn type II polynomials.

With the Equations \eqref{ortogonalidadHahnInofinal},\eqref{ortogonalidadHahnIfinal} and \eqref{ortogonalidadHahnII} we can prove the following biorthogonal relations between Hahn multiple orthogonal polynomials of type I \eqref{HahnTipoI} and type II \eqref{HahnTipoII}:
\begin{teo}
\label{biortogonalidadHahn}
For $(n_1,n_2), (m_1,m_2)\in\N_0^2$, the Hahn multiple orthogonal polynomials of type I \eqref{HahnTipoI}) and type II \eqref{HahnTipoII} satisfy the following discrete biorthogonality relations 
\begin{align*}
 \sum_{k=0}^N\left(Q_{(n_1,n_2)}^{(1)}(k)\omega_1(k)+Q_{(n_1,n_2)}^{(2)}(k)\omega_2(k)\right) Q_{(m_1,m_2)}(k)=
 \begin{cases}
 0, &\text{if $ n_1\leq m_1, n_2\leq m_2$},\\
 1, &\text{if $m_1+m_2=n_1+n_2-1$},\\
 0, &\text{if $m_1+m_2\leq n_1+n_2-2$}.
 \end{cases}
\end{align*}
\end{teo}

\begin{proof}
The case $n_1\leq m_1$, $n_2\leq m_2$ can be easily deduced from Hahn type II orthogonality relations \eqref{ortogonalidadHahnII} remembering that $\deg Q^{(a)}_{(n_1,n_2)}=n_a-1$ for $a\in\{1,2\}$. Analogously, the case $m_1+m_2\leq n_1+n_2-2$ can be easily deduced from orthogonality relations \eqref{ortogonalidadHahnInofinal} recalling that $\deg Q_{(m_1,m_2)}=m_1+m_2$.

In order to prove the remaining case we need to remind that the leading coefficient of $Q_{(m_1,m_2)}$ in the Pochhammer base $\{(-x)_j\}_{j\in\N_0}$ is $(-1)^{m_1+m_2}$. This leads to
 \begin{multline*}
   \sum_{k=0}^N\left(Q_{(n_1,n_2)}^{(1)}(k)\omega_1(k)+Q_{(n_1,n_2)}^{(2)}(k)\omega_2(k)\right) Q_{(m_1,m_2)}(k) \\\begin{aligned}
    &= \sum_{k=0}^N\left(Q_{(n_1,n_2)}^{(1)}(k)\omega_1(k)+Q_{(n_1,n_2)}^{(2)}(k)\omega_2(k)\right) (-1)^{n_1+n_2-1}(-k)_{n_1+n_2-1}\\
  &=\sum_{k=0}^N\left(Q_{(n_1,n_2)}^{(1)}(k)\omega_1(k)+Q_{(n_1,n_2)}^{(2)}(k)\omega_2(k)\right) (-N+k)_{n_1+n_2-1}=1.
   \end{aligned}
 \end{multline*}
%so that
% \begin{multline*}
% \sum_{k=0}^N\left(Q_{(n_1,n_2)}^{(1)}(k)\omega_1(k)+Q_{(n_1,n_2)}^{(2)}(k)\omega_2(k)\right) Q_{(m_1,m_2)}(k) = 
% \\ = \sum_{k=0}^N\left(Q_{(n_1,n_2)}^{(1)}(k)\omega_1(k)+Q_{(n_1,n_2)}^{(2)}(k)\omega_2(k)\right) \sum_{l=0}^{n_1+n_2-1}\dbinom{n_1+n_2-1}{l}(N-n_1-n_2+2)_{n_1+n_2-1-l}(-N+k)_l\\
% =\sum_{l=0}^{n_1+n_2-1}\dbinom{n_1+n_2-1}{l}(N-n_1-n_2+2)_{n_1+n_2-1-l} \sum_{k=0}^N\left(Q_{(n_1,n_2)}^{(1)}(k)\omega_1(k)+Q_{(n_1,n_2)}^{(2)}(k)\omega_2(k)\right) (-N+k)_l
% \\=1.
%\end{multline*}
Where the last equality is obtained using equations
%\eqref{ortogonalidadHahnInofinal} and 
\eqref{ortogonalidadHahnIfinal}.
\end{proof}

%%%%%%%%%%%%%%%%%%%%%
\subsection{Recurrence}

The type II orthogonal polynomials $Q_{(n_1,n_2)}(x)$ satisfy the near neighbors recursion relations described in Theorem \ref{Th:Recursion_Relations_II}, with recursion coefficients given by, cf. \cite{Arvesu}:
\begin{align}
 \label{bHahnpar}
 b^{(1)}&=\begin{multlined}[t][.85\textwidth]A(n_1,n_2,\alpha_1,\alpha_2,\beta,N)+A(n_2,n_1,\alpha_2,\alpha_1+1,\beta,N)
 +C(n_1+1,n_2+1,\alpha_1,\alpha_2,\beta,N)\\+D(n_1,n_2,\alpha_1,\alpha_2,\beta,N), \end{multlined}\\
 \label{cHahnpar}
 c&=\begin{multlined}[t][.85\textwidth]\big(A(n_1,n_2,\alpha_1,\alpha_2,\beta,N)+A(n_2,n_1,\alpha_2,\alpha_1+1,\beta,N)+D(n_1,n_2,\alpha_1,\alpha_2,\beta,N)\big)\\
  \times{C}(n_2,n_1+1,\alpha_2,\alpha_1,\beta,N)+A(n_1,n_2,\alpha_1,\alpha_2,\beta,N)B(n_1,n_2,\alpha_1,\alpha_2,\beta,N) , 
 \end{multlined}\\
 \label{dHahnpar}
 d^{(1)}&=A(n_1,n_2,\alpha_1,\alpha_2,\beta,N)B(n_1,n_2,\alpha_1,\alpha_2,\beta,N)C(n_1,n_2,\alpha_1,\alpha_2,\beta,N),
\end{align}
and $b^{(2)}(n_1,n_2,\alpha_1,\alpha_2)=b^{(1)}(n_2,n_1,\alpha_2,\alpha_1)$ and  $d^{(2)}(n_1,n_2,\alpha_1,\alpha_2)= d^{(1)}(n_2,n_1,\alpha_2,\alpha_1)$.
Here,
\begin{align}
\label{A}
 A&=\dfrac{n_1(n_1+n_2+\alpha_2+\beta)(n_1+n_2+\beta)(N+n_1+\alpha_1+\beta+1)}{(n_1+2n_2+\alpha_2+\beta)(2n_1+n_2+\alpha_1+\beta)(2n_1+n_2+\alpha_1+\beta+1)},\\
 B&=\dfrac{(n_1+\alpha_1-\alpha_2)(n_1+n_2+\alpha_1+\beta)(n_1+n_2+\beta-1)(N-n_1-n_2+1)}{(n_1+2n_2+\alpha_2+\beta-1)(2n_1+n_2+\alpha_1+\beta)(2n_1+n_2+\alpha_1+\beta-1)},\\
 C&=\dfrac{(n_1+\alpha_1)(n_1+n_2+\alpha_1+\beta-1)(n_1+n_2+\alpha_2+\beta-1)(N-n_1-n_2+2)}{(n_1+2n_2+\alpha_2+\beta-2)(2n_1+n_2+\alpha_1+\beta-2)(2n_1+n_2+\alpha_1+\beta-1)},\\
\label{D}
 D&=\dfrac{n_1n_2(n_1+n_2+\beta)}{(2n_1+n_2+\alpha_1+\beta+1)(n_1+2n_2+\alpha_2+\beta)}.
\end{align}
Hence, the Hahn multiple orthogonal of type I, $Q^{(a)}_{(n_1,n_2)}(x)$, $a\in\{1,2\}$, satisfy the type I near neighbor recursion relations given in Theorem \ref{Th:Recursion_Relations_I} with the recursion coefficients specified in \eqref{bHahnpar}-\eqref{dHahnpar}.

%%%%%%%%%%%%%%%%%%%%%%%
\section{Jacobi--Piñeiro multiple orthogonal polynomials of type I}

The type II  Jacobi--Piñeiro  polynomials are obtained from type II multiple Hahn polynomials from the limit, cf. \cite{AskeyII}, 
\begin{align*}
	% \label{JPIIcomolimiteHahn}
	&P_{(n_1,n_2)}(x,\alpha_1,\alpha_2,\beta)=\lim_{N\rightarrow\infty}\dfrac{(-1)^{n_1+n_2}}{(-N)_{n_1+n_2}}Q_{(n_1,n_2)}(Nx,\alpha_1,\alpha_2,\beta,N)
\end{align*}
so that we are lead to the following expression
\begin{align}
	\label{eq:JPII}
\left\{\begin{aligned}
		P_{(n_1,n_2)}&=\begin{multlined}[t][.9\textwidth]
		(-1)^{n_1+n_2}\dfrac{(\alpha_1+1)_{n_1}(\alpha_2+1)_{n_2}}{(n_1+n_2+\alpha_1+\beta+1)_{n_1}(n_1+n_2+\alpha_2+\beta+1)_{n_2}}\\
		\times\KF{1:3;1}{1:2;0}{(\alpha_1+\beta+n_1+1):(-n_2,\alpha_2+\beta+n_1+n_2+1,\alpha_1+n_1+1);(-n_1)}{(\alpha_1+1):(\alpha_2+1,\alpha_1+\beta+n_1+1);--}{x,x},
	\end{multlined}\\
	&=\begin{multlined}[t][.9\textwidth]
		(-1)^{n_1+n_2}\dfrac{(\alpha_1+1)_{n_1}(\alpha_2+1)_{n_2}}{(\alpha_1+\beta+n_1+n_2+1)_{n_1}(\alpha_2+\beta+n_1+n_2+1)_{n_2}}\\
		\times\sum_{l=0}^{n_1+n_2}
		\dfrac{(\alpha_1+\beta+n_1+1)_l}{(\alpha_1+1)_l}
		\sum_{m=0}^l\dfrac{(-n_1)_{m}(-n_2)_{l-m}}{m!(l-m)!}\dfrac{(\alpha_2+\beta+n_1+n_2+1)_{l-m}(\alpha_1+n_1+1)_{l-m}}{(\alpha_2+1)_{l-m}(\alpha_1+\beta+n_1+1)_{l-m}}x^l.
	\end{multlined}
\end{aligned}\right.
\end{align}

Observe that generalized hypergeometric expressions are also known for these monic polynomials, cf. \cite[Equation (23.3.5)]{Ismail}, so that
\begin{align*}
	P_{(n_1,n_2)}= (-1)^{n_1+n_2} \frac{(\alpha_1+1)_{n_1}(\alpha_2+1)_{n_2}}{(n_1+n_2+\alpha_1+\beta)_{n_1}(n_1+n_2+\alpha_2+\beta)_{n_2}}
	\frac{1}{(1-x)^\beta}\pFq{3}{2}{a_1-n_1-n_2-\beta,\alpha_1+n_1+1,\alpha_2+n_2+1}{\alpha_1+1,\alpha_2+1}{}.
\end{align*}
Hence, we get a reduction formula for the Kampé de Fériet series
\begin{multline*}
	\KF{1:3;1}{1:2;0}{(\alpha_1+\beta+n_1+1):(-n_2,\alpha_2+\beta+n_1+n_2+1,\alpha_1+n_1+1);(-n_1)}{(\alpha_1+1):(\alpha_2+1,\alpha_1+\beta+n_1+1);--}{x,x}\\=  \frac{1}{(1-x)^\beta}\pFq{3}{2}{a_1-n_1-n_2-\beta,\alpha_1+n_1+1,\alpha_2+n_2+1}{\alpha_1+1,\alpha_2+1}{x}.
\end{multline*}

The corresponding limits of the multiple Hahn  near neighbor recursion relations 
\begin{align*}
%\label{bJPcomolimiteHahn}
b_{\operatorname{JP}}^{(a)}(n_1,n_2,\alpha_1,\alpha_2,\beta)&=\lim_{N\rightarrow\infty}\dfrac{b^{(a)}(n_1,n_2,\alpha_1,\alpha_2,\beta,N)}{N},\\
%\label{cJPcomolimiteHahn}
c_{\operatorname{JP}}(n_1,n_2,\alpha_1,\alpha_2,\beta)&=\lim_{N\rightarrow\infty}\dfrac{c(n_1,n_2,\alpha_1,\alpha_2,\beta,N)}{N^2},
\\
%\label{dJPcomolimiteHahn}
d_{\operatorname{JP}}^{(a)}(n_1,n_2,\alpha_1,\alpha_2,\beta)&=\lim_{N\rightarrow\infty}\dfrac{d^{(a)}(n_1,n_2,\alpha_1,\alpha_2,\beta,N)}{N^3},
\end{align*}
leads to
\begin{align}\label{eq:RecursionJP}
\left\{\begin{aligned}
	b_{\operatorname{JP}}^{(1)}&=\dfrac{A(n_1,n_2,\alpha_1,\alpha_2,\beta,N)}{N+\alpha_1+\beta+n_1+1}+\dfrac{A(n_2,n_1,\alpha_2,\alpha_1+1,\beta,N)}{N+\alpha_2+\beta+n_2+1}+\dfrac{C(n_1+1,n_2+1,\alpha_1,\alpha_2,\beta,N)}{N-n_1-n_2},\\
c_{\operatorname{JP}}=&
\begin{multlined}[t][.82\textwidth]\left(\dfrac{A(n_1,n_2,\alpha_1,\alpha_2,\beta,N)}{N+\alpha_1+\beta+n_1+1}+\dfrac{A(n_2,n_1,\alpha_2,\alpha_1+1,\beta,N)}{N+\alpha_2+\beta+n_2+1}\right)
	\dfrac{C(n_2,n_1+1,\alpha_2,\alpha_1,\beta,N)}{N-n_2-n_1+1}\\
	+\dfrac{A(n_1,n_2,\alpha_1,\alpha_2,\beta,N)B(n_1,n_2,\alpha_1,\alpha_2,\beta,N)}{(N+\alpha_1+\beta+n_1+1)(N-n_1-n_2+1)},
\end{multlined}\\
d_{\operatorname{JP}}^{(1)}&=\dfrac{A(n_1,n_2,\alpha_1,\alpha_2,\beta,N)B(n_1,n_2,\alpha_1,\alpha_2,\beta,N)C(n_1,n_2,\alpha_1,\alpha_2,\beta,N)}{(N+\alpha_1+\beta+n_1+1)(N-n_1-n_2+1)(N-n_1-n_2+2)},
\end{aligned}\right.
\end{align}
With $b_{\operatorname{JP}}^{(2)}(n_1,n_2,\alpha_1,\alpha_2,\beta)=b_{\operatorname{JP}}^{(1)}(n_2,n_1,\alpha_2,\alpha_1,\beta)$ and $d_{\operatorname{JP}}^{(2)}(n_1,n_2,\alpha_1,\alpha_2,\beta)=d_{\operatorname{JP}}^{(1)}(n_2,n_1,\alpha_2,\alpha_1,\beta)$.

For  the type I Hahn multiple polynomials \eqref{HahnTipoI}, $a\in\{1,2\}$,  let us consider the corresponding limit process 
\begin{align}
	\label{JPIcomolimiteHahn}
	P^{(a)}_{(n_1,n_2)}(x,\alpha_1,\alpha_2,\beta)
	&=\lim_{N\rightarrow\infty}\dfrac{(-1)^{n_a-1}}{(-N)_{n_a-1}}\dfrac{1}{\Gamma(\alpha_a+1)
		\Gamma(\beta+1)}\dfrac{\Gamma(\alpha_a+\beta+N+n_a+1)}{(N+1-n_1-n_2)!}Q^{(a)}_{(n_1,n_2)}(Nx,\alpha_1,\alpha_2,\beta,N).
\end{align}

Coefficients \eqref{eq:RecursionJP} are the near neighbor recursion coefficients of the Jacobi--Piñeiro in \cite{Aptekarev} and the  limit of  the Hahn linear forms  $\mathscr Q_{(1,0)}$ and $\mathscr Q_{(1,1)}$ coincide with those linear forms of the Jacobi--Piñeiro case. Hence, Proposition \ref{pro:recurrence_unicity_typeI} implies  that these limit polynomials satisfy continuous orthogonality relations of the form
\begin{align*}
% \label{ortogonalidadJP}
 \int^1_{0}x^j\left(P^{(1)}_{(n_1,n_2)}(x,\alpha_1,\alpha_2,\beta)\omega^{\operatorname{JP}}_1(x,\alpha_1)
 +P^{(2)}_{(n_1,n_2)}(x,\alpha_1,\alpha_2,\beta)\omega^{\operatorname{JP}}_2(x,\alpha_2)\right)\d\mu^{\operatorname{JP}}(x)=0,
\end{align*}
for $j\in\{0,\dots,n_1+n_2-2\}$, with respect to the weight functions and measure
\begin{align*}
 %\label{pesosJP}
 \omega^{\operatorname{JP}}_a(x,\alpha_,\beta)&=x^{\alpha_a}, & a&\in \{1,2\}, & \d\mu^{\operatorname{JP}}(x)&=(1-x)^{\beta}\d x.
\end{align*}
Applying the limit \eqref{JPIcomolimiteHahn} we find:
\begin{pro}[Hypergeometric expressions for type I Jacobi--Piñeiro]
	 	For $a\in\{1,2\}$, the  Jacobi--Piñeiro multiple orthogonal polynomials of type I biorthogonal to the Jacobi--Piñeiro multiple orthogonal polynomials of type II in Equation \eqref{eq:JPII}  are
\begin{align*}
% \label{JPI}
%\begin{aligned}
  P^{(a)}_{(n_1,n_2)}&=
 \begin{multlined}[t][.9\textwidth]
  (-1)^{n_1+n_2-1}\dfrac{(\alpha_1+\beta+n_1+n_2)_{n_1}({\alpha}_2+\beta+n_1+n_2)_{{n}_2}}{(n_a-1)!(\hat{\alpha}_a-\alpha_a)_{\hat{n}_a}}
 \dfrac{\Gamma(\alpha_a+\beta+n_1+n_2)}{\Gamma(\beta+n_1+n_2)\Gamma(\alpha_a+1)}\\
 \times\pFq{3}{2}{-n_a+1,\alpha_a+\beta+n_1+n_2,\alpha_a-\hat{\alpha}_a-\hat{n}_a+1}{\alpha_a+1,\alpha_a-\hat{\alpha}_a+1}{x},
 \end{multlined}\\
 &=\begin{multlined}[t][.9\textwidth]
  (-1)^{n_1+n_2-1}\dfrac{(\alpha_1+\beta+n_1+n_2)_{n_1}({\alpha}_2+\beta+n_1+n_2)_{{n}_2}}{(n_a-1)!(\hat{\alpha}_a-\alpha_a)_{\hat{n}_a}}
 \dfrac{\Gamma(\alpha_a+\beta+n_1+n_2)}{\Gamma(\beta+n_1+n_2)\Gamma(\alpha_a+1)}\\
 \times\sum_{l=0}^{n_a-1}\dfrac{(-n_a+1)_l}{l!}\dfrac{(\alpha_a+\beta+n_1+n_2)_l(\alpha_a-\hat{\alpha}_a-\hat{n}_a+1)_l}{(\alpha_a+1)_l(\alpha_a-\hat{\alpha}_a+1)_l}x^l.
 \end{multlined}
%\end{aligned}
\end{align*}
\end{pro}
\begin{rem}
	These formulas  extend the ones found  in \cite[section 4.2]{JP} in the step-line.
\end{rem}

\

\enlargethispage{1cm}
%%%%%%%%%%%%%%%%%%%%%%%
\section{Meixner I multiple orthogonal polynomials of type I}

The type II Meixner I multiple orthogonal  polynomials appear from type II multiple Hahn polynomials  when the limit, cf. \cite{AskeyII}, 
\begin{align*}
	% \label{Meixner1KindcomolimiteHahnII}
	M_{(n_1,n_2)}(x,\beta,c_1,c_2)=\lim_{N\rightarrow\infty}Q_{(n_1,n_2)}\left(x,c_1N,c_2N,-N,-\beta\right)
\end{align*}
is taken. Then, one is lead to the following expressions
\begin{align}
	\label{Meixner1KindII}
\left\{\begin{aligned}
		M_{(n_1,n_2)}(x,\beta,c_1,c_2)&=\left(\dfrac{c_1}{c_1-1}\right)^{n_1}\left(\dfrac{c_2}{c_2-1}\right)^{n_2}(\beta)_{n_1+n_2}
	\KF{1:1;1}{1:0;0}{(-x):(-n_2);(-n_1)}{(\beta):--;--}{\dfrac{c_2-1}{c_2},\dfrac{c_1-1}{c_1}}\\
	&=\left(\dfrac{c_1}{c_1-1}\right)^{n_1}\left(\dfrac{c_2}{c_2-1}\right)^{n_2}(\beta)_{n_1+n_2}
	\sum_{l=0}^{n_1+n_2}\sum_{m=0}^l\dfrac{(-n_1)_m(-n_2)_{l-m}}{(\beta)_lm!(l-m)!}\left(\dfrac{c_1-1}{c_1}\right)^m\left(\dfrac{c_2-1}{c_2}\right)^{l-m}(-x)_l.
\end{aligned}\right.
\end{align}
The corresponding limit of the Hahn near neighbor recursion coefficients is
\begin{align*}
	b_{\operatorname{MI}}^{a)}(n_1,n_2,\beta,c_1,c_2)&=\lim_{N\rightarrow\infty} b^{(a)}\left(n_1,n_2,c_1N,c_2N,-N,-\beta\right),\\
	c_{\operatorname{MI}}(n_1,n_2,\beta,c_1,c_2)&=\lim_{N\rightarrow\infty} c\left(n_1,n_2,c_1N,c_2N,-N,-\beta\right),\\
	d_{\operatorname{MI}}^{(a)}(n_1,n_2,\beta,c_1,c_2)&=\lim_{N\rightarrow\infty} d^{(a)}\left(n_1,n_2,c_1N,c_2N,-N,-\beta\right),
\end{align*}
so that
\begin{align}\label{eq:recursionMI}
\left\{\begin{gathered}
	\begin{aligned}
		b_{\operatorname{MI}}^{(1)}&=n_1\dfrac{1+c_1}{1-c_1}+n_2\left(\dfrac{c_1}{1-c_1}+\dfrac{c_2}{1-c_2}+1\right)+\dfrac{c_1}{1-c_1}\beta,&
	b_{\operatorname{MI}}^{(2)}&=n_2\dfrac{1+c_2}{1-c_2}+n_1\left(\dfrac{c_1}{1-c_1}+\dfrac{c_2}{1-c_2}+1\right)+\dfrac{c_2}{1-c_2}\beta,
\end{aligned}\\
	c_{\operatorname{MI}}=(\beta+n_1+n_2-1)\left(n_1\dfrac{c_1}{(1-c_1)^2}+n_2\dfrac{c_2}{(1-c_2)^2}\right),\\
\begin{aligned}
		d_{\operatorname{MI}}^{(1)}&=\dfrac{n_1(\beta+n_1+n_2-2)(\beta+n_1+n_2-1)c_1(c_1-c_2)}{(1-c_1)^3(1-c_2)},&
	d_{\operatorname{MI}}^{(2)}&=\dfrac{n_2(\beta+n_1+n_2-2)(\beta+n_1+n_2-1)c_2(c_2-c_1)}{(1-c_1)(1-c_2)^3}.
\end{aligned}
\end{gathered}\right.
\end{align}

For $a\in\{1,2\}$, let us discuss the corresponding limit process over the Hahn multiple orthogonal polynomials of type I in Equation \eqref{HahnTipoI}
\begin{multline}
	\label{Meixner1KindcomolimiteHahnI}
	M_{(n_1,n_2)}^{(a)}(x,\beta,c_1,c_2)=\dfrac{(1-c_a)^{n_1+n_2-1+\beta}}{\Gamma(-\beta+1)}\\\times\lim_{N\rightarrow\infty}(-N+1)_{n_1+n_2-1}
	\dfrac{\Gamma((c_a-1)N+n_a+1-\beta)}{\Gamma((c_a-1)N+n_1+n_2+n_a)} Q^{(a)}_{(n_1,n_2)}\left(x,c_1N,c_2N,-N,-\beta\right).
\end{multline}
Coefficients in \eqref{eq:recursionMI} coincide with the ones at \cite{Arvesu} and the  limit of  the Hahn linear forms  $\mathscr Q_{(1,0)}$ and $\mathscr Q_{(1,1)}$ coincide with those linear forms of the Meixner I case. Therefore, using Proposition \ref{pro:recurrence_unicity_typeI} we conclude that these polynomials satisfy discrete orthogonality relations of the form
\begin{align*}
% \label{ortogonalidadJP}
 \sum^\infty_{k=0}(k+\beta)_j\left(M_{(n_1,n_2)}^{(1)}(k,\beta,c_1,c_2)\omega^{\operatorname{MI}}_1(k,\beta,c_1)
 +M_{(n_1,n_2)}^{(2)}(k,\beta,c_1,c_2)\omega^{\operatorname{MI}}_2(k,\beta,c_2)\right)=0,
\end{align*}
for $j=0,\dots,n_1+n_2-2$, with respect to the weight functions
\begin{align*}
% \label{pesoMeixner1Kind}
 \omega^{\operatorname{MI}}_a(x,\beta,c_a)&=\dfrac{\Gamma(\beta+x)c_a^x}{\Gamma(\beta)\Gamma(x+1)}, & a&\in\{1,2 \} .
\end{align*}
In this case the parameter region is $\beta>0$, $0<c_1,c_2<1$ and $c_1\neq c_2$.
Consequently, polynomials \eqref{Meixner1KindcomolimiteHahnI} are the Meixner I multiple orthogonal polynomials of type I. Explicitly, applying the limit, we find
\begin{pro}[Hypergeometric expressions for type I multiple Meixner I]
		For $a\in\{1,2\}$, the  Meixner I multiple orthogonal polynomials of type I biorthogonal to the Meixner I multiple orthogonal polynomials of type II in Equation \eqref{Meixner1KindII} are
	\begin{align*}
% \label{Meixner1KindTipoI}
%\begin{aligned}
 M_{(n_1,n_2)}^{(a)}
 &=\begin{multlined}[t][.9\textwidth]  
\dfrac{(-1)^{{n}_a-1}(n_1+n_2-2)!(1-c_a)^{n_1+n_2-1+\beta}}{(n_1-1)!(n_2-1)!(\beta)_{n_1+n_2-1}}\left(\dfrac{1-\hat{c}_a}{c_a-\hat{c}_a}\right)^{n_1+n_2-1}\\
\times \KF{2:1;0}{1:1;0}{(-n_a+1,\beta): (-x);--}{(-n_1-n_2+2): ( \beta); --}{\dfrac{(c_a-\hat{c}_a)(1-c_a)}{c_a(1-\hat{c}_a)},\dfrac{c_a-\hat{c}_a}{1-\hat{c}_a}}
 \end{multlined}\\
 &=\begin{multlined}[t][.9\textwidth]
  \dfrac{(-1)^{{n}_a-1}(n_1+n_2-2)!(1-c_a)^{n_1+n_2-1+\beta}}{(n_1-1)!(n_2-1)!(\beta)_{n_1+n_2-1}}\left(\dfrac{1-\hat{c}_a}{c_a-\hat{c}_a}\right)^{n_1+n_2-1}\\\times\sum_{l=0}^{n_a-1}\sum_{m=0}^{n_a-1-l}
 \dfrac{(-n_a+1)_{l+m}(\beta)_{l+m}}{l!m!(-n_1-n_2+2)_{l+m}( \beta)_l}\left(\dfrac{(c_a-\hat{c}_a)(1-c_a)}{c_a(1-\hat{c}_a)}\right)^l\left(\dfrac{c_a-\hat{c}_a}{1-\hat{c}_a}\right)^m(-x)^l.
 \end{multlined}
%\end{aligned}
\end{align*}
	Where we have defined $\hat{c}_a\coloneq \delta_{a,1}c_2+\delta_{a,2}c_1$.
\end{pro}

\begin{rem}
 To the best of our knowledge these explicit expressions were not known.
\end{rem}

%For the step-line, if we apply this limit over the Hahn recurrence relation \eqref{recurrenciaIHahn}, we find the polynomials $M_{(n_1,n_2)}^{(a)}$ satisfy a recurrence relation of the same form respect to the new coefficients 
%\begin{align*}
%%\label{bMeixner1KindcomolimiteHahn}
%b_n^M(\beta,c_1,c_2)&=\lim_{N\rightarrow\infty} b_n\left(c_1N,c_2N,-N,-\beta\right),\\
%%\label{cMeixner1KindcomolimiteHahn}
%c_n^M(\beta,c_1,c_2)&=\lim_{N\rightarrow\infty} c_n\left(c_1N,c_2N,-N,-\beta\right),\\
%%\label{dMeixner1KindcomolimiteHahn}
%d_n^M(\beta,c_1,c_2)&=\lim_{N\rightarrow\infty} d_n\left(c_1N,c_2N,-N,-\beta\right).
%\end{align*}
%These new coefficients are exactly the Meixner I recurrence relation coefficients given at \cite[section 4.2]{Arvesu}, this is
%\begin{align*}
%%\label{bMeixner1Kindpar}
% b^M_{2m}&=m\dfrac{1+c_1}{1-c_1}+m\left(\dfrac{c_1}{1-c_1}+\dfrac{c_2}{1-c_2}+1\right)+\dfrac{c_1}{1-c_1}\beta,\\
%% \label{bMeixner1Kindimpar}
%b^M_{2m+1}&=m\dfrac{1+c_2}{1-c_2}+(m+1)\left(\dfrac{c_1}{1-c_1}+\dfrac{c_2}{1-c_2}+1\right)+\dfrac{c_2}{1-c_2}\beta,\\
%% \label{cMeixner1Kindpar} 
% c^M_{2m}&=(\beta+2m-1)\left(m\dfrac{c_1}{(1-c_1)^2}+m\dfrac{c_2}{(1-c_2)^2}\right),\\
% % \label{cMeixner1Kindimpar} 
% c^M_{2m+1}&=(\beta+2m)\left((m+1)\dfrac{c_1}{(1-c_1)^2}+m\dfrac{c_2}{(1-c_2)^2}\right),\\
%% \label{dMeixner1Kindpar} 
% d^M_{2m}&=\dfrac{m(\beta+2m-2)(\beta+2m-1)c_1(c_1-c_2)}{(1-c_1)^3(1-c_2)},\\
% % \label{dMeixner1Kindimpar}
% d^M_{2m+1}&=\dfrac{m(\beta+2m-1)(\beta+2m)c_2(c_2-c_1)}{(1-c_1)(1-c_2)^3}.
%\end{align*}

%%%%%%%%%%%%%%%%%%%%%%%
\section{Meixner II multiple orthogonal polynomials of type I}

The type II Meixner II multiple orthogonal  polynomials can be gotten from type II multiple Hahn polynomials performing the following limit, cf. \cite{AskeyII}, 
\begin{align*}
	% \label{Meixner2KindcomolimiteHahnII}
	M_{(n_1,n_2)}(x;\beta_1,\beta_2,c)=\lim_{N\rightarrow\infty}Q_{(n_1,n_2)}\left(x;\beta_1-1,\beta_2-1,\dfrac{1-c}{c}N,N\right).
\end{align*}
and corresponding type II polynomials are found to be 
\begin{align}
	\label{Meixner2KindII}
\left\{\begin{aligned}
		M_{(n_1,n_2)}
	&=\left(\dfrac{c}{c-1}\right)^{n_1+n_2}(\beta_1)_{n_1}(\beta_2)_{n_2}\KF{1:1;2}{1:0;1}{(-x):(-n_1);(-n_2,\beta_1+n_1)}{(\beta_1):--;(\beta_2)}{\dfrac{c-1}{c},\dfrac{c-1}{c}}\\
	&=\left(\dfrac{c}{c-1}\right)^{n_1+n_2}(\beta_1)_{n_1}(\beta_2)_{n_2}\sum^{n_1+n_2}_{l=0}\sum^{l}_{m=0}\dfrac{(-n_1)_m(-n_2)_{l-m}(\beta_1+n_1)_{l-m}}{(\beta_1)_{l}(\beta_2)_{l-m} m!(l-m)!}\left(\dfrac{c-1}{c}\right)^{l}(-x)_{l}.
\end{aligned}\right.
\end{align}
The limit for the Hahn near neighbor recursion coefficients  
\begin{align*}
	b_{\operatorname{MII}}^{(a)}(n_1,n_2,\beta_1,\beta_2,c)&=\lim_{N\rightarrow\infty}b^{(a)}
	\left(n_1,n_2,\beta_1-1,\beta_2-1,\frac{1-c}{c}N,N\right),\\
	c_{\operatorname{MII}}(n_1,n_2,\beta_1,\beta_2,c)&=\lim_{N\rightarrow\infty}c\left(n_1,n_2,\beta_1-1,\beta_2-1,\frac{1-c}{c}N,N\right),\\
	d_{\operatorname{MII}}^{(a)}(n_1,n_2,\beta_1,\beta_2,c)&=\lim_{N\rightarrow\infty}d^{(a)}\left(n_1,n_2,\beta_1-1,\beta_2-1,\frac{1-c}{c}N,N
	\right),
\end{align*}
delivers
\begin{align}
	\label{eq:recursionMII}
\left\{\begin{gathered}
	\begin{aligned}
		b_{\operatorname{MII}}^{(1)}&=n_1+n_2+\dfrac{c}{1-c}(\beta_1+n_1+n_2+n_1),&
	b_{\operatorname{MII}}^{(2)}&=n_1+n_2+\dfrac{c}{1-c}(\beta_2+n_1+n_2+n_2),
\end{aligned}\\
	c_{\operatorname{MII}}=\dfrac{c}{(1-c)^2}(n_1n_2+n_1(n_1+\beta_1-1)+n_2(n_2+\beta_2-1)),\\
\begin{aligned}
		d_{\operatorname{MII}}^{(1)}&=\dfrac{c^2}{(1-c)^3}n_1(n_1+\beta_1-1)(n_1+\beta_1-\beta_2),&
	d_{\operatorname{MII}}^{(2)}&=\dfrac{c^2}{(1-c)^3}n_2(n_2+\beta_2-1)(n_2+\beta_2-\beta_1).
\end{aligned}
\end{gathered}\right.
\end{align}
For the Hahn polynomials \eqref{HahnTipoI}  we consider the corresponding limit process  
\begin{multline}
	\label{Meixner2KindcomolimiteHahnI}
	M_{(n_1,n_2)}^{(a)}(x;\beta_1,\beta_2,c)=\dfrac{(1-c)^{\beta_a+n_1+n_2+n_a-2}}{c^{n_1+n_2-1}}\\
	\times\lim_{N\rightarrow\infty}\dfrac{(\beta_a+\frac{1-c}{c}N+n_1+n_2+n_a-1)_{N+1-n_1-n_2}}{(N+1-n_1-n_2)!}Q_{(n_1,n_2)}^{(a)}\left(x;\beta_1-1,\beta_2-1,\dfrac{1-c}{c}N,N\right),
\end{multline}
Coefficients  \eqref{eq:recursionMII} are precisely the coefficients in \cite{Arvesu} and  the  limit of  the Hahn linear forms  $\mathscr Q_{(1,0)}$ and $\mathscr Q_{(1,1)}$ coincide with those linear forms of the Meixner II case. Therefore, from Proposition \ref{pro:recurrence_unicity_typeI} we deduce that
these polynomials satisfy discrete orthogonality relations of the form
\begin{align*}
%\label{ortogonalidadJP}
\sum^\infty_{k=0}(-k)_j\left(M^{(1)}_{(n_1,n_2)}(k,\beta_1,\beta_2,c)\omega^{\operatorname{MII}}_1(k,\beta_1,c)
+M_{(n_1,n_2)}^{(2)}(k,\beta_1,\beta_2,c)\omega^{\operatorname{MII}}_2(k,\beta_2,c)\right)=0,
\end{align*}
for $j\in\{0,\dots,n_1+n_2-2\}$, with respect to the weight functions
\begin{align*}
 % \label{pesosMeixner2Kind}
 \omega^{\operatorname{MII}}_a(x,\beta_a,c)&=\dfrac{\Gamma(\beta_a+x)c^x}{\Gamma(\beta_a)\Gamma(x+1)}, & a&\in\{1,2\}.
\end{align*}
This time we have that $\beta_1,\beta_2>0$; $0<c<1$ and $\beta_1-\beta_2\not\in\mathbb Z$.
So, we have that \eqref{Meixner2KindcomolimiteHahnI} are the multiple Meixner II polynomials of type I. 
For $a\in\{1,2\}$, the limit \eqref{Meixner2KindcomolimiteHahnI} gives:
\begin{pro}[Hypergeometric expressions for type I multiple Meixner II]
			For $a\in\{1,2\}$, the  Meixner II multiple orthogonal polynomials of type I biorthogonal to the Meixner II multiple orthogonal polynomials of type II in Equation \eqref{Meixner2KindII} are
	\begin{align*}
% \label{Meixner2KindI}
 M_{(n_1,n_2)}^{(a)}
 &=\begin{multlined}[t][.9\textwidth]
  \dfrac{(1-c)^{\beta_a+n_1+n_2+n_a-2}}{c^{n_1+n_2-1}}\dfrac{(-1)^{n_a-1}(n_1+n_2-2)!}{(n_1-1)!(n_2-1)!}\dfrac{1}{(\beta_a-\hat{\beta}_a-\hat{n}_a+1)_{n_1+n_2-1}}\\
 \times\KF{1:2;1}{1:1;0}{(-n_a+1):(-x,\,\beta_a-\hat{\beta}_a-\hat{n}_a+1);(\hat{\beta}_a-\beta_a-n_a+1)}{(-n_1-n_2+2):(\beta_a);--}{1,\dfrac{c}{c-1}},
 \end{multlined}\\
 &=\begin{multlined}[t][.9\textwidth]
  \dfrac{(1-c)^{\beta_a+n_1+n_2+n_a-2}}{c^{n_1+n_2-1}}\dfrac{(-1)^{n_a-1}(n_1+n_2-2)!}{(n_1-1)!(n_2-1)!}\dfrac{1}{(\beta_a-\hat{\beta}_a-\hat{n}_a+1)_{n_1+n_2-1}}\\
 \times\sum^{n_a-1}_{l=0}\dfrac{1}{l!}\dfrac{(\beta_a-\hat{\beta}_a-\hat{n}_a+1)_{l}}{(\beta_a)_l}\sum^{n_a-1-l}_{m=0}\dfrac{(-n_a+1)_{l+m}}{m!}\dfrac{(\hat{\beta}_a-\beta_a-n_a+1)_{m}}{(-n_1-n_2+2)_{l+m}}\left(\dfrac{c}{c-1}\right)^m(-x)_l.
 \end{multlined}
\end{align*}
Where we have defined $\hat{\beta}_a\coloneq \beta_1\delta_{a,2}+\beta_2\delta_{a,1}$. 
\end{pro}
\begin{rem}
	These explicit expressions are to our best knowledge new.
\end{rem}

%If we apply this limit over the Hahn recurrence relation \eqref{recurrenciaIHahn}, we find the that the polynomials $M_{(n_1,n_2)}^{(a)}$ satisfy a recurrence relation of the same form respect to the new coefficients
%\begin{align*}
%% \label{bMeixner2KindcomolimiteHahn}
% &b_n^M(\beta_1,\beta_2,c)=\lim_{N\rightarrow\infty}b^Q_n\left(\beta_1-1,\beta_2-1,\frac{1-c}{c}N,N\right)\\
%% \label{cMeixner2KindcomolimiteHahn}
% &c_n^M(\beta_1,\beta_2,c)=\lim_{N\rightarrow\infty}c^Q_n\left(\beta_1-1,\beta_2-1,\frac{1-c}{c}N,N\right)\\
%% \label{dMeixner2KindcomolimiteHahn}
% &d_n^M(\beta_1,\beta_2,c)=\lim_{N\rightarrow\infty}d^Q_n\left(\beta_1-1,\beta_2-1,\frac{1-c}{c}N,N\right)
%\end{align*}
%These new coefficients are exactly the Meixner II recurrence relation coefficients given at \cite[section 4.3]{Arvesu}, this is
%\begin{align*}
%% \label{bMeixner2Kindpar}
% b^{M}_{2m}(\beta_1,\beta_2,c)&=2m+\dfrac{c}{1-c}(\beta_1+3m), &b^M_{2m+1}(\beta_1,\beta_2,c)&=2m+1+\dfrac{c}{1-c}(\beta_2+3m+1),\\
% c^M_{2m}(\beta_1,\beta_2,c)&=\dfrac{c}{(1-c)^2}m(\beta_1+\beta_2+3m-2), &c^M_{2m+1}(\beta_1,\beta_2,c)&=\dfrac{c}{(1-c)^2}((m+1)\beta_1+m(\beta_2+3m+1)),\\
% d^M_{2m}(\beta_1,\beta_2,c)&=\dfrac{c^2}{(1-c)^3}m(m+\beta_1-1)(m+\beta_1-\beta_2), &
%% \label{dMeixner2Kindimpar}
% d^M_{2m+1}(\beta_1,\beta_2,c)&=\dfrac{c^2}{(1-c)^3}m(m+\beta_2-1)(m+\beta_2-\beta_1).
%\end{align*}

%%%%%%%%%%%%%%%%%%%%%%%
\section{Kravchuk multiple orthogonal polynomials of type I}

The type II Kravchuk multiple orthogonal  polynomials are derived  from type II multiple Hahn polynomials from the limit, cf. \cite{AskeyII}, 
\begin{align*}
	% \label{KravchukcomolimiteHahnII}
	K_{(n_1,n_2)}(x,p_1,p_2,N)=\lim_{t\rightarrow\infty}Q_{(n_1,n_2)}\left(x,\dfrac{p_1}{1-p_1}t,\dfrac{p_2}{1-p_2}t,t,N\right),
\end{align*}
and they read as follows
\begin{align}
	\label{KravchukII}
\left\{\begin{aligned}
		K_{(n_1,n_2)}
	&=p_1^{n_1}p_2^{n_2}(-N)_{n_1+n_2}
	\KF{1:1;1}{1:0;0}{(-x):(-n_2);(-n_1)}{(-N):--;--}{\dfrac{1}{p_2},\dfrac{1}{p_1}}
	\\ 
	&=p_1^{n_1}p_2^{n_2}(-N)_{n_1+n_2}\sum_{l=0}^{n_1+n_2}\sum_{m=0}^l\dfrac{(-n_1)_m(-n_2)_{l-m}}{m!(l-m)!}\dfrac{1}{p_1^m}\dfrac{1}{p_2^{l-m}}\dfrac{(-x)_l}{(-N)_l}.
\end{aligned}\right.
\end{align}
The  limit in the Hahn near neighbor recursion coefficients  
\begin{align*}
	b_{\operatorname{K}}^{(a)}(n_1,n_2,p_1,p_2,N)&=\lim_{t\rightarrow\infty} b^{(a)}\left(n_1,n_2,\dfrac{p_1}{1-p_1}t,\dfrac{p_2}{1-p_2}t,t,N\right),\\
	c_{\operatorname{K}}^(n_1,n_2,p_1,p_2,N)&=\lim_{t\rightarrow\infty} c\left(n_1,n_2,\dfrac{p_1}{1-p_1}t,\dfrac{p_2}{1-p_2}t,t,N\right),\\
	d_{\operatorname{K}}^{(a)}(n_1,n_2,p_1,p_2,N)&=\lim_{t\rightarrow\infty} d^{(a)}\left(n_1,n_2,\dfrac{p_1}{1-p_1}t,\dfrac{p_2}{1-p_2}t,t,N\right).
\end{align*}
leads to
\begin{align}\label{eq:recursionK}
\left\{\begin{aligned}
		b_{\operatorname{K}}^{(1)}&=n_1+n_2+(N-n_1-n_2-n_1)p_1-n_2p_2,\\
	b_{\operatorname{K}}^{(2)}&=n_1+n_2-n_1p_1+(N-n_1-n_2-n_2)p_2,\\
	c_{\operatorname{K}}&=(N-n_1-n_2+1)(n_1p_1(1-p_1)+n_2p_2(1-p_2)),\\
	d_{\operatorname{K}}^{(1)}&=(N-n_1-n_2+1)(N-n_1-n_2+2)n_1p_1(1-p_1)(p_1-p_2),\\
	d_{\operatorname{K}}^{(2)}&=(N-n_1-n_2+1)(N-n_1-n_2+2)n_2p_2(1-p_2)(p_2-p_1).
\end{aligned}\right.
\end{align}

For $a\in\{1,2\}$, we discuss for the Hahn polynomials \eqref{HahnTipoI} the corresponding limit process  
\begin{multline}
	\label{KravchukcomolimiteHahnI}
	K_{(n_1,n_2)}^{(a)}(x,p_1,p_2,N)=%\begin{multlined}[t][.5\textwidth]
		\dfrac{1}{N!}\left(\dfrac{1}{1-p_a}\right)^{n_1+n_2-1}\\
		\times \lim_{t\rightarrow\infty}(t+1)_{n_1+n_2-1}\left(\dfrac{1}{1-p_a}t+n_1+n_2+n_a\right)_{N+1-n_1-n_2}
		Q_{(n_1,n_2)}^{(a)}\left(x,\dfrac{p_1}{1-p_1}t,\dfrac{p_2}{1-p_2}t,t,N\right).
%	\end{multlined}
\end{multline}
Coefficients in \eqref{eq:recursionK} coincide with the coefficients given in \cite{Arvesu} and the  limit of  the Hahn linear forms  $\mathscr Q_{(1,0)}$ and $\mathscr Q_{(1,1)}$ also coincide with those linear forms of the Kravchuk case. Hence, recalling Proposition \ref{pro:recurrence_unicity_typeI}, these polynomials satisfy discrete orthogonality relations of the form
\begin{align*}
% \label{ortogonalidadJP}
 \sum^N_{k=0}(-N+k)_j\left(K_{(n_1,n_2)}^{(1)}(k,p_1,p_2,N)\omega^{\operatorname{K}}_1(k,p_1,N)
 +K_{(n_1,n_2)}^{(2)}(k,p_1,p_2,N)\omega^{\operatorname{K}}_2(k,p_2,N)\right)=0
\end{align*}
for $j\in\{0,\dots,n_1+n_2-2\}$, with respect to the weight functions
\begin{align*}
% \label{pesosKravchuk}
 \omega^{\operatorname{K}}_a(x,p_a,N)&=\dfrac{N!p_a^x(1-p_a)^{N-x}}{\Gamma(x+1)\Gamma(N-x+1)}, & a&\in\{1,2\}.
\end{align*}
In this case we have that
$N\in \mathbb N_0$; $0<p_1,p_2<1$ and $p_1\neq p_2$. 
So we have that (\ref{KravchukcomolimiteHahnI}) are  the multiple Kravchuk polynomials of type I. 
For $a\in\{1,2\}$, when we compute the limit \eqref{KravchukcomolimiteHahnI} we get:
\begin{pro}[Hypergeometric expressions for type I multiple Kravchuk]
				For $a\in\{1,2\}$, the  Kravchuk multiple orthogonal polynomials of type I biorthogonal to the Kravchuk multiple orthogonal polynomials of type II in Equation \eqref{KravchukII} are
	\begin{align*}
% \label{KravchukTipoI}
%\begin{aligned}
  K_{(n_1,n_2)}^{(a)}
  &=\begin{multlined}[t][.75\textwidth]
  \dfrac{(-1)^{n_a-1}(n_1+n_2-2)!}{(n_1-1)!(n_2-1)!(N-n_1-n_2+2)_{n_1+n_2-1}(p_a-\hat{p}_a)^{n_1+n_2-1}}\\\times\KF{2:1;0}{1:1;0}{(-n_a+1,-N): (-x); --}{(-n_1-n_2+2): (-N); --}{\dfrac{p_a-\hat{p}_a}{p_a(1-{p}_a)},\dfrac{\hat{p}_a-p_a}{1-{p}_a}}
 \end{multlined}\\
 &=\begin{multlined}[t][.75\textwidth]\dfrac{(-1)^{n_a-1}(n_1+n_2-2)!}{(n_1-1)!(n_2-1)!(N-n_1-n_2+2)_{n_1+n_2-1}(p_a-\hat{p}_a)^{n_1+n_2-1}}\\
 \times\sum_{l=0}^{n_a-1}\sum_{m=0}^{n_a-1-l}\dfrac{(-n_a+1)_{l+m}(-N)_{l+m}}{l!m!(-n_1-n_2+2)_{l+m}(-N)_l}\left(\dfrac{p_a-\hat{p}_a}{p_a(1-{p}_a)}\right)^l\left(\dfrac{\hat{p}_a-p_a}{1-{p}_a}\right)^m(-x)_l.
  \end{multlined}
%\end{aligned}
\end{align*}
Where we have defined $\hat{p}_a\coloneq p_1\delta_{a,2}+p_2\delta_{a,1}$. 
\end{pro}
\begin{rem}
	These expressions where not presented before in the literature.
\end{rem}

%If we apply this limit over the Hahn recurrence relation \eqref{recurrenciaIHahn}, we find the polynomials $K_{(n_1,n_2)}^{(a)}$ satisfy a recurrence relation of the same form respect to the new coefficients
%\begin{align*}
%%\label{bKravchukcomolimiteHahn}
%b_n^K(p_1,p_2,N)&=\lim_{t\rightarrow\infty} b_n\left(\dfrac{p_1}{1-p_1}t,\dfrac{p_2}{1-p_2}t,t,N\right),\\
%%\label{cKravchukcomolimiteHahn}
%c_n^K(p_1,p_2,N)&=\lim_{t\rightarrow\infty} c_n\left(\dfrac{p_1}{1-p_1}t,\dfrac{p_2}{1-p_2}t,t,N\right),\\
%%\label{dKravchukcomolimiteHahn}
%d_n^K(p_1,p_2,N)&=\lim_{t\rightarrow\infty} d_n\left(\dfrac{p_1}{1-p_1}t,\dfrac{p_2}{1-p_2}t,t,N\right).
%\end{align*}
%These new coefficients are exactly the Kravchuk recurrence relation coefficients given at \cite[section 4.4]{Arvesu}, this is
%
%\begin{align*}
%%\label{bKravchukpar}
% b^K_{2m}&=2m+(N-3m)p_1-mp_2, &
% b^K_{2m+1}&=2m+1-(m+1)p_1+(N-3m-1)p_2,\\
% c^K_{2m}&=(N-2m+1)m(p_1(1-p_1)+p_2(1-p_2)), &
% c^K_{2m+1}&=(N-2m)(m+1)p_1(1-p_1)+mp_2(1-p_2),\\
% d^K_{2m}&=(N-2m+1)(N-2m+2)mp_1(1-p_1)(p_1-p_2), & 
% d^K_{2m+1}&=(N-2m)(N-2m+1)mp_2(1-p_2)(p_2-p_1).
%\end{align*}

%%%%%%%%%%%%%%%%%%%%%%%
\section{Laguerre I multiple orthogonal polynomials of type I}

Laguerre I multiple orthogonal polynomials of type II appear from the following two limits, cf. \cite{AskeyII}, 
\begin{align*}
	% \label{Laguerre1KindcomolimiteJPII}
	L_{(n_1,n_2)}(x,\alpha_1,\alpha_2)&=\lim_{\beta\rightarrow\infty}(\alpha_1+\beta+n_1+n_2)_{n_1}(\alpha_2+\beta+n_1+n_2)_{n_2}P_{(n_1,n_2)}\left(\dfrac{x}{\beta};\alpha_1,\alpha_2,\beta\right),\\
	% \label{Laguerre1KindcomolimiteMeixner2KindII}
	&=\lim_{c\rightarrow1}\left({1-c}\right)^{n_1+n_2}M_{(n_1,n_2)}\left(\dfrac{x}{1-c};\alpha_1+1,\alpha_2+1,c\right),
\end{align*}
and in both limits  the following expressions are found
\begin{align}
	\label{Laguerre1KindII}
\left\{\begin{aligned}
		L_{(n_1,n_2)}
	&=(-1)^{n_1+n_2}(\alpha_1+1)_{n_1}(\alpha_2+1)_{n_2}\KF{0:2;1}{1:1;0}{--:(-n_2,\alpha_1+n_1+1);(-n_1)}{(\alpha_1+1):(\alpha_2+1);--}{x,x}
	\\
	&=(-1)^{n_1+n_2}(\alpha_1+1)_{n_1}(\alpha_2+1)_{n_2}\sum^{n_1+n_2}_{l=0}\sum^{l}_{m=0}\dfrac{(-n_1)_m(-n_2)_{l-m}(\alpha_1+1+n_1)_{l-m}}{(\alpha_1+1)_{l}(\alpha_2+1)_{l-m} m!(l-m)!}x^l.
\end{aligned}\right.
\end{align}
These limits lead to the following limits from the coefficients in the near neighbor recursion relation of Jacobi--Piñeiro type
\begin{align*}
	b_{\operatorname{LI}}^{(a)}(n_1,n_2,\alpha_1,\alpha_2)&=\lim_{\beta\rightarrow\infty}\beta b_{\operatorname{JP}}^{(a)}(n_1,n_2,\alpha_1,\alpha_2,\beta)=\lim_{c\rightarrow1}(1-c)b_{\operatorname{MII}}^{(a)}(n_1,n_2,\alpha_1+1,\alpha_2+1,c),\\
	c_{\operatorname{LI}}(n_1,n_2,\alpha_1,\alpha_2)&=\lim_{\beta\rightarrow\infty}\beta^2c_{\operatorname{JP}}(n_1,n_2,\alpha_1,\alpha_2,\beta)=\lim_{c\rightarrow1}{(1-c)^2}c_{\operatorname{MII}}(n_1,n_2,\alpha_1+1,\alpha_2+1,c),\\
	d_{\operatorname{LI}}^{(a)}(n_1,n_2,\alpha_1,\alpha_2)&=\lim_{\beta\rightarrow\infty}\beta^3d_{\operatorname{JP}}^{(a)}(n_1,n_2,\alpha_1,\alpha_2,\beta)=\lim_{c\rightarrow1}{(1-c)^3}d_{\operatorname{MII}}^{(a)}(n_1,n_2,\alpha_1+1,\alpha_2+1,c),
\end{align*}
that is 
\begin{align}\label{eq:recursionLI}
\left\{\begin{gathered}
		\begin{aligned}
		b_{\operatorname{LI}}^{(1)}&=n_1+n_2+n_1+\alpha_1+1,
		&b_{\operatorname{LI}}^{(2)}&=n_1+n_2+n_2+\alpha_2+1,
	\end{aligned}\\
	c_{\operatorname{LI}}=n_1n_2+n_1(n_1+\alpha_1)+n_2(n_2+\alpha_2), \\
	\begin{aligned}
		d_{\operatorname{LI}}^{(1)}&=n_1(n_1+\alpha_1)(n_1+\alpha_1-\alpha_2), &
		d_{\operatorname{LI}}^{(2)}&=n_2(n_2+\alpha_2)(n_2+\alpha_2-\alpha_1).
	\end{aligned}
\end{gathered}\right.
\end{align}

The Laguerre I type I polynomials can be obtained applying the same procedure as before to the Jacobi--Piñeiro polynomials or the Meixner II polynomials  doing the respective limit processes:
\begin{align}
\label{Laguerre1KindcomolimiteJPI}
 L_{(n_1,n_2)}^{(a)}(x;\alpha_1,\alpha_2)&=\lim_{\beta\rightarrow\infty}\dfrac{1}{(\alpha_1+\beta+n_1+n_2)_{n_1}(\alpha_2+\beta+n_1+n_2)_{n_2}}\dfrac{\Gamma(\beta+n_1+n_2)}{\Gamma(\alpha_a+\beta+n_1+n_2)} P_{(n_1,n_2)}^{(a)}\left(\frac{x}{\beta};\alpha_1,\alpha_2,\beta\right)\\
\label{Laguerre1KindcomolimiteMeixner2KindI}
 % L_{(n_1,n_2)}^{(a)}(x;\alpha_1,\alpha_2)
 &=\lim_{c\rightarrow1}\dfrac{1}{(1-c)^{\alpha_a+n_1+n_2}\Gamma(\alpha_a+1)}M_{(n_1,n_2)}^{(a)}\left(\frac{x}{1-c};\alpha_1+1,\alpha_2+1,c\right).
\end{align}

As coefficients \eqref{eq:recursionLI}  are the near neighbors recursion coefficients for the Laguerre I multiple orthogonality, cf. \cite{Aptekarev}, and the  limit of  the Jacobi--Piñeiro or Meixner II linear forms  $\mathscr Q_{(1,0)}$ and $\mathscr Q_{(1,1)}$ coincide with corresponding  linear forms of the Laguerre I case, from Proposition \ref{pro:recurrence_unicity_typeI} it follows that, for $j\in\{0,\dots,n_1+n_2-2\}$, these polynomials satisfy continuous orthogonality relations of the form
\begin{align*}
% \label{ortogonalidadLaguerreI}
 \int^{\infty}_{0}x^j\left(L_{(n_1,n_2)}^{(1)}(x,\alpha_1,\alpha_2)\omega^{\operatorname{LI}}_1(x,\alpha_1)
 +L_{(n_1,n_2)}^{(2)}(x,\alpha_1,\alpha_2)\omega^{\operatorname{LI}}_2(x,\alpha_2)\right)\d\mu^{\operatorname{LI}}(x)=0
\end{align*}
with respect to the weight functions and measure
\begin{align*}
% \label{pesosLaguerre1Kind}
 \omega^{\operatorname{LI}}_a(x,\alpha_a)&=\Exp{-x}x^{\alpha_a}, & a&\in\{1,2\}, & \d\mu^{\operatorname{LI}}(x)&=\d x.
\end{align*}
In this case we still have $\alpha_1,\alpha_2>-1$ and $\alpha_1-\alpha_2\not\in\mathbb Z$.
For $a\in\{1,2\}$, applying any of the two limits \eqref{Laguerre1KindcomolimiteJPI} or \eqref{Laguerre1KindcomolimiteMeixner2KindI} we find for the Laguerre I multiple orthogonal polynomials of type I the following expressions in terms of generalized hypergeometric series:
\begin{pro}[Hypergeometric expressions for type I multiple Laguerre I]
		For $a\in\{1,2\}$, the Laguerre I multiple orthogonal polynomials of type I biorthogonal to the Laguerre I multiple orthogonal polynomials of type II in Equation \eqref{Laguerre1KindII} are
	\begin{align*}
% \label{Laguerre1KindI}
 L_{(n_1,n_2)}^{(a)}
 &=(-1)^{n_1+n_2-1}\dfrac{1}{(n_a-1)!\Gamma(\alpha_a+1)(\hat{\alpha}_a-\alpha_a)_{\hat{n}_a}}
 \pFq{2}{2}{-n_a+1,\alpha_a-\hat{\alpha}_a-\hat{n}_a+1}{\alpha_a+1,\alpha_a-\hat{\alpha}_a+1}{x}\\
 &=(-1)^{n_1+n_2-1}\dfrac{1}{(n_a-1)!\Gamma(\alpha_a+1)(\hat{\alpha}_a-\alpha_a)_{\hat{n}_a}}\sum^{n_a-1}_{l=0}\dfrac{(-n_a+1)_l}{l!}\dfrac{(\alpha_a-\hat{\alpha}_a-\hat{n}_a+1)_l}{(\alpha_a+1)_l(\alpha_a-\hat{\alpha}_a+1)_l}x^l.
\end{align*}
\end{pro}
\begin{rem}
To the best of our knowledge these hypergeometric formulas are completely new. 
\end{rem}

%
%The recurrence coefficients transform as
%\begin{align*}
%%\label{bLaguerre1KindcomolimiteJP}
% b^L_n(\alpha_1,\alpha_2)&=\lim_{\beta\rightarrow\infty}\beta b^P_n(\alpha_1,\alpha_2,\beta)=\lim_{c\rightarrow1}(1-c)b^M_n(\alpha_1+1,\alpha_2+1,c)\\
%%\label{cLaguerre1KindcomolimiteJP}
% c^L_n(\alpha_1,\alpha_2)&=\lim_{\beta\rightarrow\infty}\beta^2c^P_n(\alpha_1,\alpha_2,\beta)=\lim_{c\rightarrow1}{(1-c)^2}c^M_n(\alpha_1+1,\alpha_2+1,c)\\
%%\label{dLaguerre1KindcomolimiteJP}
% d^L_n(\alpha_1,\alpha_2)&=\lim_{\beta\rightarrow\infty}\beta^3d^P_n(\alpha_1,\alpha_2,\beta)=\lim_{c\rightarrow1}{(1-c)^3}d^M_n(\alpha_1+1,\alpha_2+1,c)
%\end{align*}
%recovering the Laguerre I coefficients given at \cite[section 3.2]{Clasicos}
%\begin{align*}
%%\label{bLaguerre1Kindpar}
% b^L_{2m}&=3m+1+\alpha_1,
% &b^L_{2m+1}&=3m+2+\alpha_2,\\
% c^L_{2m}&=m(3m+\alpha_1+\alpha_2), &
% c^L_{2m+1}&=3m^2+m(\alpha_1+\alpha_2+3)+\alpha_1+1,\\
% d^L_{2m}&=m(m+\alpha_1)(m+\alpha_1-\alpha_2), &
%%\label{dLaguerre1Kindimpar}
% d^L_{2m+1}&=m(m+\alpha_2)(m+\alpha_2-\alpha_1).
%\end{align*}

%%%%%%%%%%%%%%%%%%%%%%%%%%%
\section{Laguerre II multiple orthogonal polynomials of type I}

Laguerre II multiple orthogonal polynomials of type II appear from multiple Meixner II polynomials by the limit , cf. \cite{AskeyII}, 
\begin{align*}
	% \label{Laguerre2KindcomolimiteMeixner1KindII}
	L_{(n_1,n_2)}(x;\alpha_0,c_1,c_2)=\lim_{t\rightarrow\infty}\dfrac{1}{t^{n_1+n_2}}M_{(n_1,n_2)}\left({tx};\alpha_0+1,\dfrac{t}{t+c_1},\dfrac{t}{t+c_2}\right)
\end{align*}
and read
\begin{align}
	\label{LaguerreII}
\left\{\begin{aligned}
		L_{(n_1,n_2)}&=(-1)^{n_1+n_2}\dfrac{(\alpha_0+1)_{n_1+n_2}}{c_1^{n_1}c_2^{n_2}}\KF{0:1;1}{1:0;0}{--:(-n_1);(-n_2)}{(\alpha_0+1):--;--}{c_1x,c_2x}\\
	&=(-1)^{n_1+n_2}\dfrac{(\alpha_0+1)_{n_1+n_2}}{c_1^{n_1}c_2^{n_2}}\sum_{l=0}^{n_1+n_2}\sum_{m=0}^l\dfrac{(-n_1)_m}{m!}\dfrac{(-n_2)_{l-m}}{(l-m)!}\dfrac{c_1^mc_2^{l-m}}{(\alpha_0+1)_l}x^l.
\end{aligned}\right.
\end{align}
The Meixner I near neighbor recursion coefficients limit is
\begin{align*}
	b_{\operatorname{LII}}^{(a)}(n_1,n_2,\alpha_0,c_1,c_2)&=\lim_{t\rightarrow\infty}\dfrac{1}{t}b_{\operatorname{MI}}^{(a)}\left(n_1,n_2,\alpha_0+1,\dfrac{t}{t+c_1},\dfrac{t}{t+c_2}\right),\\
	c_{\operatorname{LII}}(n_1,n_2,\alpha_0,c_1,c_2)&=\lim_{t\rightarrow\infty}\dfrac{1}{t^2}c_{\operatorname{MI}}\left(n_1,n_2,\alpha_0+1,\dfrac{t}{t+c_1},\dfrac{t}{t+c_2}\right),\\
	d_{\operatorname{LII}}^{(a)}(n_1,n_2,\alpha_0,c_1,c_2)&=\lim_{t\rightarrow\infty}\dfrac{1}{t^3}d_{\operatorname{MI}}^{(a)}\left(n_1,n_2,\alpha_0+1,\dfrac{t}{t+c_1},\dfrac{t}{t+c_2}\right),
\end{align*}
with
\begin{align}\label{eq:recursionLII}
\left\{\begin{gathered}
	\begin{aligned}
		b_{\operatorname{LII}}^{(1)}&=\dfrac{n_2c_1+2n_1c_2+n_2c_2+\alpha_0c_2+c_2}{c_1c_2},&
	b_{\operatorname{LII}}^{(2)}&=\dfrac{n_1c_2+2n_2c_1+n_1c_1+\alpha_0 c_1+c_1}{c_1c_2},
\end{aligned}\\
	c_{\operatorname{LII}}=\dfrac{(n_1c_2^2+n_2c_1^2)(n_1+n_2+\alpha_0)}{c^2_1c^2_2},\\ 
\begin{aligned}
		d_{\operatorname{LII}}^{(1)}&=\dfrac{n_1(c_2-c_1)(n_1+n_2+\alpha_0-1)(n_1+n_2+\alpha_0)}{c^3_1c_2},&
	d_{\operatorname{LII}}^{(2)}&=\dfrac{n_2(c_1-c_2)(n_1+n_2+\alpha_0-1)(n_1+n_2+\alpha_0)}{c_1c^3_2}.
\end{aligned}
\end{gathered}\right.
\end{align}

The Laguerre II polynomials of type I can be obtained from the Meixner I polynomials %(\ref{Meixner1KindTipoI}) 
through the limit
\begin{align}
	\label{Laguerre2KindcomolimiteMeixner1KindI}
	L_{(n_1,n_2)}^{(a)}(x;\alpha_0,c_1,c_2)&=\lim_{t\rightarrow\infty}\dfrac{t^{n_1+n_2-1}(c_i+t)^{\alpha_0+1}}{\Gamma(\alpha_0+1)}M_{(n_1,n_2)}^{(a)}\left(tx,\alpha_0+1,\dfrac{t}{t+c_1},\dfrac{t}{t+c_2}\right).
\end{align}
Coefficients \eqref{eq:recursionLII} coincide with those of the multiple Laguerre II type,  cf. \cite{Aptekarev}, and the  limit of  the Jacobi--Piñeiro or Meixner I linear forms  $\mathscr Q_{(1,0)}$ and $\mathscr Q_{(1,1)}$ coincide with corresponding  linear forms of the Laguerre II case, so that  Proposition \ref{pro:recurrence_unicity_typeI} leads to  conclude that these polynomials satisfy continuous orthogonality relations of the form
\begin{align*}
% \label{ortogonalidadLaguerreI}
 \int^{\infty}_{0}x^j\left(L_{(n_1,n_2)}^{(1)}(x,\alpha_0,c_1,c_2)\omega^{\operatorname{LII}}_1(x,\alpha_0,c_1)
 +L_{(n_1,n_2)}^{(2)}(x,\alpha_0,c_1,c_2)\omega^{\operatorname{LII}}_2(x,\alpha_0,c_2)\right)\d\mu^{\operatorname{LII}}(x)=0
\end{align*}
for $j\in\{0,\dots,n_1+n_2-2\}$, in terms of the weight functions and measure
\begin{align*}
% \label{pesosLaguerre2Kind}
 \omega_a^{\operatorname{LII}}(x,\alpha_0,c_a)&=\Exp{-c_ax}x^{\alpha_0}, & a&\in\{1,2\}, & \d\mu^{\operatorname{LII}}(x)&=\d x,
\end{align*}
with $\alpha_0>-1$; $c_1,c_2>0$ and $c_1\neq c_2$.
For $a\in\{1,2\}$, Equation \eqref{Laguerre2KindcomolimiteMeixner1KindI} leads to the following explicit expressions of Laguerre II multiple orthogonal polynomials of type I:
\begin{pro}[Hypergeometric expressions for type I multiple Laguerre II]
		For $a\in\{1,2\}$, the  Laguerre II multiple orthogonal polynomials of type I biorthogonal to the Laguerre II multiple orthogonal polynomials of type II in Equation \eqref{LaguerreII} are
	\begin{align*}
% \label{Laguerre2KindTipoI}
 L_{(n_1,n_2)}^{(a)}(x,\alpha_0,c_1,c_2)
 &=\begin{multlined}[t][.8\textwidth] \dfrac{(-1)^{{n}_a-1}(n_1+n_2-2)!}{(n_1-1)!(n_2-1)!\Gamma(\alpha_0+n_1+n_2)}c_a^{\alpha_0+1}\left(\dfrac{c_a\hat{c}_a}{\hat{c}_a-c_a}\right)^{n_1+n_2-1}\\\times\KF{2:1;0}{1:1;0}{(-n_a+1,\alpha_0+1): --;--}{(-n_1-n_2+2): ( \alpha_0+1); --}{\dfrac{(c_a-\hat{c}_a)c_a}{\hat{c}_a}x,\dfrac{(\hat{c}_a-c_a)}{\hat{c}_a}}
\end{multlined}\\
 &=\begin{multlined}[t][.8\textwidth]
  \dfrac{(-1)^{{n}_a-1}(n_1+n_2-2)!}{(n_1-1)!(n_2-1)!\Gamma(\alpha_0+n_1+n_2)}c_a^{\alpha_0+1}\left(\dfrac{c_a\hat{c}_a}{\hat{c}_a-c_a}\right)^{n_1+n_2-1}\\
 \times\sum_{l=0}^{n_a-1}\sum_{m=0}^{n_a-1-l}\dfrac{(-n_a+1)_{l+m}(\alpha_0+1)_{l+m}}{l!m!(-n_1-n_2+2)_{l+m}( \alpha_0+1)_l}\left(\dfrac{(c_a-\hat{c}_a)c_a}{\hat{c}_a}\right)^l\left(\dfrac{(\hat{c}_a-c_a)}{\hat{c}_a}\right)^mx^l.
 \end{multlined}
\end{align*}
\end{pro}
\begin{rem}
	These expressions where not known before in the literature.
\end{rem}

%The recurrence relations coefficients transform as
%\begin{align*}
%%\label{bLaguerre2KindcomolimiteMEixne1Kind}
% b^L_n(\alpha_0,c_1,c_2)&=\lim_{t\rightarrow\infty}\dfrac{1}{t}b^M_n\left(\alpha_0+1,\dfrac{t}{t+c_1},\dfrac{t}{t+c_2}\right),\\
%%\label{cLaguerre2KindcomolimiteMEixne1Kind}
% c^L_n(\alpha_0,c_1,c_2)&=\lim_{t\rightarrow\infty}\dfrac{1}{t^2}c^M_n\left(\alpha_0+1,\dfrac{t}{t+c_1},\dfrac{t}{t+c_2}\right),\\
%%\label{dLaguerre2KindcomolimiteMEixne1Kind}
% d^L_n(\alpha_0,c_1,c_2)&=\lim_{t\rightarrow\infty}\dfrac{1}{t^3}d^M_n\left(\alpha_0+1,\dfrac{t}{t+c_1},\dfrac{t}{t+c_2}\right),
%\end{align*}
%recovering those ones given at \cite[section 3.3]{Clasicos}, this is
%\begin{align*}
%%\label{bLaguerre2Kindpar}
% b^L_{2m}&=\dfrac{m(c_1+3c_2)+c_2+\alpha_0 c_2}{c_1c_2}, &
% b^L_{2m+1}&=\dfrac{m(3c_1+c_2)+2c_1+c_2+\alpha_0 c_1}{c_1c_2}\\
% c^L_{2m}&=\dfrac{m(2m+\alpha_0)(c^2_1+c^2_2)}{c^2_1c^2_2}, &
% c^L_{2m+1}&=\dfrac{2m^2(c^2_1+c^2_2)+m(c^2_1+3c^2_2+\alpha_0(c^2_1+c^2_2))+c^2_2+\alpha_0 c^2_2}{c^2_1c^2_2}\\
% d^L_{2m}&=\dfrac{m(2m+\alpha_0)(2m+\alpha_0-1)(c_2-c_1)}{c^3_1c_2}, &
%%\label{dLaguerre2Kindimpar}
% d^L_{2m+1}&=\dfrac{m(2m+\alpha_0)(2m+\alpha_0+1)(c_1-c_2)}{c_1c^3_2}.
%\end{align*}

%%%%%%%%%%%%%%%%%%%%%%%
\section{Charlier multiple orthogonal polynomials of type I}

Charlier multiple orthogonal polynomials of type II are known to appear whenever one applies one of the two following limits, cf. \cite{AskeyII}, 
\begin{align*}
	% \label{CharliercomolimiteMeixner1KindII}
	C_{(n_1,n_2)}(x,b_1,b_2)&=\lim_{\beta\rightarrow\infty}M_{(n_1,n_2)}\left(x,\beta,\dfrac{b_1}{\beta},\dfrac{b_2}{\beta}\right)%,&
	% \label{CharliercomolimiteKravchukII}
	%C_{(n_1,n_2)}(x,a_1,a_2)&
	=\lim_{N\rightarrow\infty}K_{(n_1,n_2)}\left(x,\dfrac{b_1}{N},\dfrac{b_2}{N},N\right),
\end{align*}
and one is lead to the expression
\begin{align}
	 \label{CharlierII}
\left\{\begin{aligned}
		C_{(n_1,n_2)}
	&=(-1)^{n_1+n_2}b_1^{n_1}b_2^{n_2}\KF{1:1;1}{0:0;0}{(-x):(-n_2);(-n_1)}{--:--;--}{-\dfrac{1}{b_2},-\dfrac{1}{b_1}}\\
	&=(-1)^{n_1+n_2}b_1^{n_1}b_2^{n_2}\sum_{l=0}^{n_1+n_2}\sum_{m=0}^l\dfrac{(-n_2)_{l-m}}{(l-m)!}\dfrac{(-n_1)_m}{m!}
	\left(-\dfrac{1}{b_1}\right)^m
	\left(
	-\dfrac{1}{b_2}\right)^{l-m}
	(-x)_l.
\end{aligned}\right.
\end{align}
The corresponding limit for the near neighbors recursion coefficients is 
\begin{align*}
	b_{\operatorname{C}}^{(a)}(n_1,n_2,b_1,b_2)&=\lim_{\beta\rightarrow\infty} b_{\operatorname{MI}}^{(a)}\left(n_1,n_2,\beta,\dfrac{b_1}{\beta},\dfrac{b_2}{\beta}\right)=\lim_{N\rightarrow\infty} b_{\operatorname{K}}^{(a)}\left(n_1,n_2,\dfrac{b_1}{N},\dfrac{b_2}{N},N\right),\\
	c_{\operatorname{C}}(n_1,n_2,b_1,b_2)&=\lim_{\beta\rightarrow\infty} c_{\operatorname{MI}}\left(n_1,n_2,\beta,\dfrac{b_1}{\beta},\dfrac{b_2}{\beta}\right)=\lim_{N\rightarrow\infty} c_{\operatorname{K}}\left(n_1,n_2,\dfrac{b_1}{N},\dfrac{b_2}{N},N\right),\\
	d_{\operatorname{C}}^{(a)}(n_1,n_2,b_1,b_2)&=\lim_{\beta\rightarrow\infty} d_{\operatorname{MI}}^{(a)}\left(n_1,n_2,\beta,\dfrac{b_1}{\beta},\dfrac{b_2}{\beta}\right)=\lim_{N\rightarrow\infty} d_{\operatorname{K}}^{(a)}\left(n_1,n_2,\dfrac{b_1}{N},\dfrac{b_2}{N},N\right),
\end{align*}
so that
\begin{align}\label{eq:recursionC}
		b_{\operatorname{C}}^{(1)}&=n_1+n_2+b_1, &b_{\operatorname{C}}^{(2)}&=n_1+n_2+b_2,
&
	c_{\operatorname{C}}&=n_1b_1+n_2b_2,&
		d_{\operatorname{C}}^{(1)}&=n_1b_1(b_1-b_2), &
	d_{\operatorname{C}}^{(2)}&=n_2b_2(b_2-b_1).
\end{align}

The Charlier type I polynomials can be obtained from the Meixner I polynomials %\eqref{Meixner1KindTipoI} 
and the Kravchuk polynomials %\eqref{KravchukTipoI} 
through the respective corresponding limit processes
\begin{align}
	\label{CharliercomolimiteMeixner1KindI}
	C_{(n_1,n_2)}^{(a)}(x,b_1,b_2)&=\lim_{\beta\rightarrow\infty}M_{(n_1,n_2)}^{(a)}\left(x,\beta,\dfrac{b_1}{\beta},\dfrac{b_2}{\beta}\right)\\
	\label{CharliercomolimiteKravchukI}
	&=\Exp{-b_i}\lim_{N\rightarrow\infty}K_{(n_1,n_2)}^{(a)}\left(x,\dfrac{b_1}{N},\dfrac{b_2}{N},N\right).
\end{align}
Coefficients in \eqref{eq:recursionC} coincide with the near neighbor recursion coefficients for the multiple Charlier polynomials, cf. \cite{Arvesu}, and the  limit of  the Meixner I or Kravchuk linear forms  $\mathscr Q_{(1,0)}$ and $\mathscr Q_{(1,1)}$ coincide with corresponding  linear forms of the Charlier  case, then  Proposition \ref{pro:recurrence_unicity_typeI} imply that, for $j\in\{0,\dots,n_1+n_2-2\}$, these polynomials satisfy discrete orthogonality relations of the form
\begin{align*}
% \label{ortogonalidadJP}
 \sum^\infty_{k=0}(-k)_j\left(C_{(n_1,n_2)}^{(1)}(k,b_1,b_2)\omega^{\operatorname{JP}}_1(k,b_1)
 +C_{(n_1,n_2)}^{(2)}(k,b_1,b_2)\omega^{\operatorname{C}}_2(k,b_2)\right)=0,
\end{align*}
with respect to the weight functions
\begin{align*}
 %\label{pesosCharlier}
 \omega^{\operatorname{C}}_a(x,b_a)&=\frac{b_a^x}{\Gamma(x+1)}, & a&\in\{1,2\},
\end{align*}
with $b_1,b_2>0$ and $b_1\neq b_2$.
For $a\in\{1,2\}$, Equations \eqref{CharliercomolimiteMeixner1KindI} or \eqref{CharliercomolimiteKravchukI} give the following explicit expressions for Charlier multiple orthogonal polynomials of type I polynomials 
\begin{pro}[Hypergeometric expressions for type I multiple Charlier]
		For $a\in\{1,2\}$, the  Charlier  multiple orthogonal polynomials of type I biorthogonal to the Charlier multiple orthogonal polynomials of type II in Equation \eqref{CharlierII} are
	\begin{align*}
% \label{CharlierTipoI}
 C_{(n_1,n_2)}^{(a)}
 &=\dfrac{(-1)^{n_a-1}(n_1+n_2-2)!}{(n_1-1)!(n_2-1)!}\dfrac{e^{-b_a}}{(b_a-\hat{b}_a)^{n_1+n_2-1}}\KF{1:1;0}{1:0;0}{(-n_a+1): (-x); --}{(-n_1-n_2+2): --; --}{\dfrac{b_a-\hat{b}_a}{b_a},{b_a-\hat{b}_a}}\\
 &=\dfrac{(-1)^{n_a-1}(n_1+n_2-2)!}{(n_1-1)!(n_2-1)!}\dfrac{e^{-b_a}}{(b_a-\hat{b}_a)^{n_1+n_2-1}}
 \sum_{l=0}^{n_a-1}\sum_{m=0}^{n_a-1-l}\dfrac{(-n_a+1)_{l+m}(-x)_l}{l!m!(-n_1-n_2+2)_{l+m}}\left(\dfrac{b_a-\hat{b}_a}{b_a}\right)^l({b_a-\hat{b}_a})^m.
\end{align*}
\end{pro}
\begin{rem}
	To the best of our knowledge these hypergeometrical expressions  are brand new.
\end{rem}

%The recurrence relation coefficients transform as
%\begin{align}
%%\label{bCharliercomolimiteMeixner1Kind}
% b_n^C(b_1,b_2)&=\lim_{\beta\rightarrow\infty} b^M_n\left(\beta,\dfrac{b_1}{\beta},\dfrac{b_2}{\beta}\right)=\lim_{N\rightarrow\infty} b^K_n\left(\dfrac{b_1}{N},\dfrac{b_2}{N},N\right)\\
%%\label{cCharliercomolimiteMeixner1Kind}
% c_n^C(b_1,b_2)&=\lim_{\beta\rightarrow\infty} c^M_n\left(\beta,\dfrac{b_1}{\beta},\dfrac{b_2}{\beta}\right)=\lim_{N\rightarrow\infty} c^K_n\left(\dfrac{b_1}{N},\dfrac{b_2}{N},N\right)\\
%%\label{dCharliercomolimiteMeixner1Kind}
% d_n^C(b_1,b_2)&=\lim_{\beta\rightarrow\infty} d^M_n\left(\beta,\dfrac{b_1}{\beta},\dfrac{b_2}{\beta}\right)=\lim_{N\rightarrow\infty} d^K_n\left(\dfrac{b_1}{N},\dfrac{b_2}{N},N\right)
%\end{align}
%recovering those ones given at \cite[section 4.1]{Arvesu}, this is
%\begin{align*}
% b^C_{2m}&=2m+b_1, &b^C_{2m+1}&=2m+1+b_2,\\
% c^C_{2m}&=m(b_1+b_2), &c^C_{2m+1}&=(m+1)b_1+mb_2,\\
% d^C_{2m}&=mb_1(b_1-b_2), &
% d^C_{2m+1}&=mb_2(b_2-b_1).
%\end{align*}

%\newpage
%%%%%%%%%%%%%%%%%%%%%%%%%%%%%%%
\section{Conclusions and Outlook}
%%%%%%%%%%%%%%%%%%%%%%%%%%%%%%%%%%%%

%So we have got explicit expressions as hypergeometric functions or Kampé de Fériet series for all the type I polynomials' families shown in the following diagram, as well as all the limit relations between them

%\begin{figure}[H]
 % \centering

% \caption{Limit relations between hypergeometric multiple type I polynomials}
% \label{AskeyII}
%\end{figure}

While multiple orthogonal polynomials is an active area of research and explicit expressions for the more easy type II orthogonal polynomials are known, is infrequently to find explicit expressions for the corresponding type I families. This is the case of the discrete Hahn multiple orthogonal polynomials of type I. 
In this paper we have succeeded in finding such an explicit expression in terms of Kampé de Fériet series. These expressions hold for any couple of nonnegative indexes, not necessarily in the step-line. Then, applying adequate limits and following a partial Askey scheme \cite{AskeyII} we also succeed in finding generalized/Kampé de Fériet hypergeometrical expressions for the multiple orthogonal polynomials of type I for the Jacobi--Piñeiro, Meixner I, Meixner II, Kravchuk, Laguerre I, Laguerre II and Charlier families.

If we look at the Askey scheme for scalar polynomials \cite{Askey} and its analogue for type II polynomials \cite{AskeyII} we will notice there should be also a limit relation between Jacobi--Piñeiro and Laguerre II type I polynomials; as well as four limit relations from Jacobi--Piñeiro, Kravchuk, Laguerre I and Charlier to Hermite type I polynomials. However, we couldn't find any of these relations nor a explicit expression for the Hermite ones. This will be a subject for a future research, as well as all the other families from the Askey scheme that can't be obtained from Hahn.

Also, all these developments will be relevant for those cases such that the recurrence relations in the step-line are represented by a positive Hessenberg tetradiagonal matrix. These can be truncated and renormalized to stochastic matrices representing corresponding non simple random walks (beyond birth and death) that we will study the corresponding Karlin--MacGregor spectral representation elsewhere.

\section*{Acknowledgments}
AB acknowledges Centro de Matemática da Universidade de Coimbra UIDB/00324/2020, funded by the Portuguese Government through FCT/MECS.

JEFD acknowledges CIDMA Center for Research and Development in Mathematics and Applications (University of Aveiro) and the Portuguese Foundation for Science and Technology (FCT) within project UIDB/04106/2020 and UIDP/04106/2020 and [PID2021- 122154NB-I00], \emph{Ortogonalidad y Aproximación con Aplicaciones en Machine Learning y Teoría de la Probabilidad}.

AF acknowledges CIDMA Center for Research and Development in Mathematics and Applications (University of Aveiro) and the Portuguese Foundation for Science and Technology (FCT) within project UIDB/04106/2020 and UIDP/04106/2020.

MM acknowledges Spanish ``Agencia Estatal de Investigación'' research projects [PGC2018-096504-B-C33], \emph{Ortogonalidad y Aproximación: Teoría y Aplicaciones en Física Matemática} and [PID2021- 122154NB-I00], \emph{Ortogonalidad y Aproximación con Aplicaciones en Machine Learning y Teoría de la Probabilidad}.

%\newpage
%%%%%%%%%%%%%%%%%%%%%%%%%%%%%%%%%
%%%%%%%%%%%%%%%%%%%%%%%%%%%%%%%%%
%%%%%%%%%%%%%%%%%%%%%%%%%%%%%%%%%%%%%%%%%%%%%%%%%%
%%%%%%%%%%%%%%%%%%%%%%%%%%%%%%%%%%%%%%%%%%%%%%%%%%


\begin{thebibliography}{99}
	%\addcontentsline{toc}{section}{References}
	
	%\bibliographystyle{%amsalpha
		%amsplain}
	%\bibliography{mybib}
	
	\bibitem{adler}
	Mark Adler, Pierre van Moerbeke, and P. Vanhaecke,
	\emph{Moment matrices and multi-component KP, with applications to random matrix theory},
	Communications in Mathematical Physics \textbf{286} (2009)
	1–38.
	
	\bibitem{afm}
	Carlos Álvarez-Fernández, Ulises Fidalgo, and Manuel Mañas,
	\emph{Multiple orthogonal polynomials of mixed type: Gauss--Borel factorization and the multi-component 2D Toda hierarchy},
	Advances in Mathematics~\textbf{227} (2011) 1451–1525.
	
	%Libro de Hypergeom
	\bibitem{LibrodeHypergeom}
	George E. Andrews, Roger Askey, and Ranjan Roy,
	\textit{Special functions},
	Cambridge University Press, Cambridge, 1999.%, https://doi.org/10.1017/CBO9781107325937
	
	\bibitem{Apery}
	Roger Apéry,
	\emph{Irrationalité de $\zeta(2)$ et $\zeta(3)$}.
	Astérisque \textbf{61} (1979) 11–13.
	
	\bibitem{Aptekarev}
	Alexander I. Aptekarev, Amílcar Branquinho, and Walter Van Assche, \emph{Multiple orthogonal polynomials for classical weights}, Transactions of the American Mathematical Society \textbf{355} (2003):10, 3887-3914.
	
	\bibitem{Aptekarev_Kaliaguine_Lopez}
	Alexander Aptekarev, Valery Kaliaguine, Guillermo López Lagomasino, and Ignacio Rocha,
	\emph{On the limit behavior of recurrence coefficients for multiple orthogonal polynomials},
	Journal of Approximation Theory \textbf{139} (2006), no. 1-2,
	346–370.
	
	\bibitem{Aptekarev_Kaliaguine_Saff}
	Alexander Aptekarev, Valery Kaliaguine, and Edward Saff,
	\emph{Higher-order three-term recurrences and asymptotics of multiple orthogonal polynomials},
	Constructive Approximation \textbf{30} (2009), no. 2,
	175–223.
	
	%Arvesu, Coussement, Van Assche
	\bibitem{Arvesu}
	Jorge Arvesú, Jonnathan Coussement, and 
	Walter Van Assche,
	\emph{Some discrete multiple orthogonal polynomials},
	Journal of Computational and Applied Mathematics \textbf{153} (2003) 19-45.
	% https://doi.org/10.1016/S0377-0427(02)00597-6
	
	\bibitem{Ball_Rivoal}
	Keith Balland and Tanguy Rivoal,
	\emph{Irrationalité d’une infinité de valeurs de la fonction zêta aux entiers impairs},
	Inventiones mathematicae \textbf{146} (2001) 193–207.
	
	%Esquema Tipo II
	\bibitem{AskeyII}
	Bernhard Beckermann, Jonathan Coussement and 
	Walter Van Assche, \emph{Multiple Wilson and Jacobi--Piñeiro polynomials}
	Journal of Approximation Theory \textbf{132} (2005) 155-181.%, https://doi.org/10.1016/j.jat.2004.12.001
	
	\bibitem{Ball_Rivoal}
	Keith Balland and Tanguy Rivoal,
	\emph{Irrationalité d’une infinité de valeurs de la fonction zêta aux entiers impairs},
	Inventiones mathematicae \textbf{146} (2001) 193–207.
	
	\bibitem{Bererzin_Kuijlaars_Parra} Sergey Berezin, Arno B. J. Kuijlaars, Iván Parra, \emph{Planar orthogonal polynomials as type I multiple orthogonal polynomials}, (2022) \hyperref{https://arxiv.org/abs/2212.06526}{}{}{\texttt{arXiv:2212.06526}}.
	
	\bibitem{Bleher_Kuijlaars}
	Pavel M. Bleher and Arno B.J. Kuijlaars,
	\emph{Random matrices with external source and multiple orthogonal polynomials},
	International Mathematics Research Notices \textbf{2004}:3 (2004)
	109–129.
	
	
	\bibitem{Hipergeometricos}
	Amílcar Branquinho, Juan E. Fernández-Díaz, Ana Foulquié-Moreno, and Manuel Mañas,
	\emph{Hypergeometric Multiple Orthogonal Polynomials and Random Walks}, (2021)
	\hyperref{https://arxiv.org/abs/2107.00770}{}{}{\texttt{arXiv:2107.00770}}.
	
	
	\bibitem{bfm}
	Amílcar Branquinho, Ana Foulquié-Moreno, and Manuel Mañas, 
	% \emph{Multiple orthogonal polynomials on the step-line}, \hyperref{https://arxiv.org/abs/2106.12707}{}{}{\texttt{arXiv:2106.12707 [CA]}}.
	\emph{Multiple orthogonal polynomials: Pearson equations and Christoffel formulas},
	Analysis and Mathematical Physics \textbf{12}:6 (2022)
	1–59.
	
	\bibitem{PBF_1}
	----------, 
	\emph{Oscillatory banded Hessenberg matrices, multiple orthogonal polynomials and random walks}, (2022) \hyperref{https://arxiv.org/abs/2203.13578}{}{}{\texttt{arXiv:2203.13578}}.
	
	
	
	\bibitem{PBF_2}
	----------,  
	\emph{Spectral theory for bounded banded matrices
		with positive bidiagonal factorization and mixed multiple orthogonal polynomials}, (2022)
	\hyperref{https://arxiv.org/abs/2212.10235}{}{}{\texttt{arXiv:2212.10235}}.
	
	\bibitem{JP}
	Amílcar Branquinho, Ana Foulquié-Moreno, Manuel Mañas, Carlos Álvarez-Fernández, and Juan E. Fernández-Díaz,
	\emph{Multiple Orthogonal Polynomials and Random Walks}, (2021)
	\hyperref{https://arxiv.org/abs/2103.13715}{}{}{\texttt{arXiv:2103.13715}}.
	
	\bibitem{Evi_Arno}
	Evi Daems and Arno B. J. Kuijlaars, \emph{Multiple orthogonal polynomials of mixed type and non-intersecting Brownian motions},
	Journal of Approximation Theory \textbf{146} (2007)
	91–114.
	
	%Ismail
	\bibitem{Ismail}
	Mourad E. H. Ismail,
	\textit{Classical and Quantum Orthogonal Polynomials in One Variable},
	Cambridge University Press, Cambdridge, 2005.
	
	\bibitem{Kalyagin}
	Valery Kalyagin, 
	\emph{Hermite--Padé Approximants and Spectral Analysis of Nonsymmetric Operators},
	Sbornik: Mathematics \textbf{82} (1995) 199–216.
	
	\bibitem{Kaliaguine}
	%Valery Kaliaguine, 
	----------,
	\emph{The operator moment problem, vector continued fractions and an explicit form of the Favard theorem for vector orthogonal polynomials}, 
	Journal of Computational and Applied Mathematics \textbf{65} (1995) 181–193.
	
	\bibitem{KaliaguineII}
	----------, 
	\emph{On operators associated with Angelesco systems},
	East Journal on Approximations \textbf{2} (1995)
	157–170.
	
	%Karp--Prilepkina
	\bibitem{KP}
	Dimitri B. Karp and Elena G. Prilepkina,
	\emph{Extensions of Karlsson–Minton summation
		theorem and some consequences of the first
		Miller–Paris transformation},
	Integral Transforms and Special Functions \textbf{29}:12 (2018) 1-16.
	%, \url{https://doi.org/10.1080/10652469.2018.1526793}
	
	%Esquema de Askey escalar
	\bibitem{Askey}
	Roelof Koekoek and René F. Swarttouw,
	\emph{The Askey-scheme of hypergeometric orthogonal polynomials
		and its q-analogue}, \hyperref{https://arxiv.org/abs/math/9602214}{}{}{\texttt{arXiv:9602214}}.
	%\hyperref{https://arxiv.org/abs/math/9602214}
	
	\bibitem{Lee} Seung-Yeop Lee and Meng Yang, \emph{Planar orthogonal polynomials as type II multiple orthogonal polynomials}, Journal of  Physics A: Mathematical \& Theoretical  \textbf{52}   (2019) 275202.
	
	%Más múltiples
	\bibitem{masmultiples}
	Marjolein Leurs and Walter Van Assche,
	\emph{Laguerre--Angelesco multiple orthogonal polynomials
		on an r-star}, Journal of Approximation Theory \textbf{250} (2020) 105324.%,https://doi.org/10.1016/j.jat.2019.105324
	
	
	%Lima y Loureiro
	\bibitem{LL}
	Helder Lima and Ana Loureiro,
	\emph{Multiple orthogonal polynomials with respect
		to Gauss’ hypergeometric function
	}, Studies in Applied Mathematics (2021) 154-185.
	%https://doi.org/10.1111/sapm.12437
	
	
	\bibitem{andrei_walter}
	Andrei Martínez-Finkelshtein and Walter Van Assche,
	\emph{What is a multiple orthogonal polynomial?},
	Notices of the AMS \textbf{63}:9 (2016)
	1029–1031. 
	
\bibitem{Nikiforov_Suslov_Uvarov}	Arnold F. Nikiforov, Sergei K. Suslov, and Vasilii B. Uvarov,\emph{ Classical Orhogonal Polynomials of a Discrete Variable}, Springer Series in Computational Physics, Springer, Berlin, 1991.
	
	\bibitem{nikishin_sorokin}
	Evgenii M. Nikishin and Vladimir N. Sorokin,
	\emph{Rational Approximations and Orthogonality},
	Translations of Mathematical Monographs \textbf{92},
	American Mathematical Society, Providence, 1991.
	
	\bibitem{RR}
	Medhat A. Rakha and Arjun K. Rathie,
	\emph{Certain reduction and transformation
		formulas for the Kampé de Fériet function},
	Communications of the Korean Mathematical Society \textbf{37}:2 (2022) 473–496,
	%\url{https://doi.org/10.4134/CKMS.c210090}
	
	
	%Libro de KF
	\bibitem{LibrodeKF}
	Lucy J. Slater,
	\textit{Generalized hypergeometric functions},
	Cambridge University Press, Cambridge, 2008%, https://doi.org/10.1017/CBO9781107325937
	
	\bibitem{Srivastava_Karlsson} Hari M. Srivastava and Per W. Karlsson, \emph{Multiple Gaussian Hypergeometric Series}, Ellis Horwood Limited, Hohn Wiley \& Sons, Chichester, 1985.
	
	
	\bibitem{Clasicos}
	Walter Van Assche and Els Coussement,
	\textit{Some classical multiple orthogonal polynomials},
	Journal od Computational and Applied Mathematics \textbf{127} (2001) 317-347. 
	%https://doi.org/10.1016/S0377-0427(00)00503-3
	
	
	\bibitem{VanAssche_Wolfs} Walter Van Assche and Thomas Wolfs, \emph{Multiple orthogonal polynomials associated with the exponential integral},
	\hyperref{https://arxiv.org/abs/2211.04858}{}{}{\texttt{arXiv:2211.04858}}.
	
	\bibitem{Zudilin}
	Vadim V. Zudilin,
	\emph{One of the numbers $ \zeta(5), \zeta(7), \zeta(9), \zeta(11)$ is irrational},
	Russian Mathematica Surveys \textbf{56}:4 (2001) 774–776.
	
\end{thebibliography}
\end{document}